\documentclass[a4paper, 10pt]{article} 
\usepackage[preprint]{optional}
\usepackage[margin=3.5cm]{geometry}

\usepackage{amsmath,amssymb}
\opt{ieee}{
\interdisplaylinepenalty=2500
}
\usepackage{amsthm,enumerate,verbatim}
\usepackage{amsfonts}
\usepackage{graphicx}
\usepackage{algorithm}
\usepackage{algorithmic}
\usepackage{url}
\usepackage{paralist}

\opt{preprint}{
\usepackage{lmodern}
}

\opt{mydraft}{
\usepackage{refcheck}
\usepackage{pdfsync}
}

\graphicspath{{figures/}}

\newtheorem{lemma}{Lemma}

\newtheorem{theorem}{Theorem}

\newtheorem{remark}{Remark}

\DeclareMathOperator{\trace}{trace}

\DeclareMathOperator{\argmin}{argmin}
\DeclareMathOperator{\rank}{rank}
\DeclareMathOperator{\diag}{diag}

\newcommand{\norm}[1]{\lVert #1 \rVert}

\newcommand{\R}{\ensuremath{\mathbb R}}
\newcommand{\Rnn}{\ensuremath{\mathbb R}_+}

\newcommand{\defby}{\mathrel{\mathop:}=}

\newcommand{\bydef}{=\mathrel{\mathop:}}

\newcommand{\setcond}{\ensuremath{\mid}}

\newcommand{\restrict}{\ensuremath{|}}

\newcommand{\bigo}[1]{\ensuremath{\mathcal{O}\left(#1\right)}}

\providecommand{\card}[1]{\ensuremath{\lvert#1\rvert}}

\newcommand{\email}[1]{{\ttfamily #1}}

\graphicspath{{figures/}}

\newcommand{\keywordstring}{Nonnegative matrix factorization,
separability, sparse regression, self dictionary, fast gradient,
hyperspectral imaging, pure-pixel assumption.}

\title{A Fast Gradient Method for Nonnegative Sparse Regression with Self Dictionary} 

\date{}

\author{
Nicolas Gillis\thanks{Department of Mathematics and Operational Research,
    Facult\'e Polytechnique, Universit\'e de Mons,
    Rue de Houdain 9, 7000 Mons, Belgium,
    (\email{nicolas.gillis@umons.ac.be}). NG acknowledges the support of the F.R.S.-FNRS (incentive grant for scientific research n$^\text{o}$ F.4501.16) and of the ERC (starting grant n$^\text{o}$ 679515). }\opt{ieee}{,}
\opt{preprint}{\and}
Robert Luce\thanks{\'Ecole Polytechnique F\'ed\'erale de Lausanne,
    Station 8, 1015 Lausanne, Switzerland
    (\email{robert.luce@epfl.ch})}
}
\begin{document}

\maketitle

\begin{abstract}

A nonnegative matrix factorization (NMF) can be computed efficiently
under the \emph{separability assumption}, which asserts that all the
columns of the given input data matrix belong to the cone generated by
a (small) subset of them.  The provably most robust methods to
identify these \emph{conic basis columns} are based on nonnegative
sparse regression and self dictionaries, and require the solution of
large-scale convex optimization problems.  In this paper we study a
particular nonnegative sparse regression model with self dictionary.
As opposed to previously proposed models, this model yields a smooth
optimization problem where the sparsity is enforced through linear
constraints.  We show that the Euclidean projection on the polyhedron
defined by these constraints can be computed efficiently, and propose
a fast gradient method to solve our model.  We compare our algorithm
with several state-of-the-art methods on synthetic data sets and
real-world hyperspectral images.

\end{abstract}

\opt{preprint}{
\paragraph*{Keywords} \keywordstring
}

\opt{ieee}{
\begin{IEEEkeywords}
\keywordstring
\end{IEEEkeywords}
}

\section{Introduction}
\label{sec:intro}

Given a matrix $M \in \mathbb{R}^{m,n}$ where each column of
$M$ represents a point in a data set, we assume in this paper that each
data point can be well approximated using a nonnegative linear combination
of a small subset of the data points.
More precisely, we assume that there exists a small subset
$\mathcal{K} \subset \{1,2,\dotsc,n\}$ of $r$ column indices and a
nonnegative matrix $H \in \Rnn^{r,n}$ such that
\begin{equation*}
    M \approx M(:,\mathcal{K}) H.
\end{equation*}
If $M$ is nonnegative, this problem is closely
related to nonnegative matrix factorization (NMF) which aims at
decomposing $M$ as the product of two nonnegative matrices $W \in
\mathbb{R}^{m,r}_+$ and $H \in \mathbb{R}^{r,n}_+$ with
$r \ll \min(m,n)$ such that $M \approx WH$~\cite{LS99}.
In the NMF literature, the assumption above is referred to as the
\emph{separability assumption}~\cite{AGKM11}, and the aim is therefore
to finding a particular NMF with $W = M(:,\mathcal{K})$.

There are several applications to solving near-separable NMF, e.g.,
blind hyperspectral unmixing~\cite{Jose12, Ma14},
topic modeling and document classification~\cite{KSK12, Ar13},
video summarization and image classification~\cite{ESV12},
and blind source separation~\cite{CMCW08, CMCW11, LHKL16}.

\subsection{Self Dictionary and Sparse Regression based Approaches}

Many algorithms have been proposed recently
to solve the near-separable NMF problem;
see~\cite{G14a}
and the references therein. In hyperspectral unmixing,
the most widely used algorithms sequentially identify important columns of $M$, such as vertex component analysis (VCA)~\cite{ND05} or the successive projection algorithm (SPA)~\cite{MC01, GV14};
see~section~\ref{sec:exp} for more details.
Another important class of algorithms for
identifying a good subset $\mathcal{K}$ of the columns of $M$
is based on sparse regression and self dictionaries. These algorithms are computationally more expensive but
have the advantage to consider the selection of the indices in $\mathcal{K}$ at once leading to the most robust algorithms;
see the discussion~\cite{G14a}.

An exact model for nonnegative sparse regression with self
dictionary~\cite{EMO12, ESV12} is
\begin{equation}
\label{eq:sreg}
\min_{X \in \Rnn^{n,n}} \; \norm{X}_{\text{row},0} \; \text{such that} \;
        \norm{{M} - {M}X} \leq \epsilon, 
\end{equation}
where $\norm{X}_{\text{row},0}$ equals to the number of nonzero rows
of $X$, and $\epsilon$ denotes the noise level of the data $M$.  Here,
the norm in which the residual $M - MX$ is measured should be chosen
in dependence of the noise model.  It can be checked that there is
a nonnegative matrix $X$ with $r$ nonzero rows such that $M  = MX$ if
and only if there exists an index set $\mathcal{K}$ of cardinality $r$
and a nonnegative matrix $H$ such that $M = M(:,\mathcal{K})H$: The
index set $\mathcal{K}$ corresponds to the indices of the nonzero rows
of $X$, and hence we have $H = X(\mathcal{K},:)$;
see~\cite[sec.~3]{BRRT12} or~\cite[sec.~I-B]{EMO12} for
details. 

In \cite{EMO12, ESV12}, the difficult problem \eqref{eq:sreg} is
relaxed to the convex optimization problem
\begin{equation}
\label{eq:esser}
\min_{X \in \Rnn^{n,n}} \; \norm{X}_{1,q} \; \text{s.t.} \;
        \norm{{M} - {M}X} \leq \epsilon \text{ and } X \leq 1,
\end{equation}
where $\norm{X}_{1,q} \defby \sum_{i = 1}^n \norm{X(i,:)}_{q}$.  In
\cite{EMO12}, $q = +\infty$ is used while, in \cite{ESV12}, $q = 2$ is
used.  The quantity $\norm{X}_{1,q}$ is the $\ell_1$-norm of the
vector containing the $\ell_q$ norms of the rows of $X$.  Because the
$\ell_1$ norm promotes sparsity, this model is expected to generate a
matrix $X$ with only a few nonzero rows.  The reason is that the
$\ell_1$ norm is a good surrogate for the $\ell_0$ norm on the
$\ell_{\infty}$ ball.  In fact, the $\ell_1$ norm is the convex
envelope of the $\ell_0$ norm on the $\ell_{\infty}$ ball, that is,
the $\ell_1$ norm is the largest convex function smaller than the
$\ell_0$ norm on the $\ell_{\infty}$ ball; see~\cite{RFP10}.  Hence
for $q = +\infty$ and $X \leq 1$, we have $\norm{X}_{1,\infty} \leq
\norm{X}_{\text{row},0}$ so that \eqref{eq:esser} provides a lower
bound for~\eqref{eq:sreg}. 
In practice, the constraint $X \leq 1$ is often satisfied; for example, 
in hyperspectral imaging, the entries of $X$ represent abundances 
which are smaller than one. 
If this assumption is not satisfied and the input matrix is nonnegative, 
it can be normalized so that the entries of the columns of matrix $H$ (hence $X$) are at most one, as
suggested for example in~\cite{EMO12}.  This can be achieved by
normalizing each column of $M$ so that its entries sum to one.  After
such a normalization, we have for all $j$
\begin{equation*}
\begin{split}
    1
    & = \norm{M(:,j)}_1
    = \norm{M X(:,j)}_1
    = \norm{\sum_{k} M(:,k) X(k,j)}_1\\
    & = \sum_{k \in \mathcal{K}} X(k,j) \norm{M(:,k)}_1
    = \norm{X(:,j)}_1,
\end{split}
\end{equation*}
since $M$ and $X$ are nonnegative.

The model \eqref{eq:esser} was originally proved to be robust to
noise, but only at the limit, that is, only for $\epsilon \rightarrow
0$, and assuming no columns of $M(:,\mathcal{K})$ are repeated in the
data set \cite{EMO12}.  If a column of $M(:,\mathcal{K})$ is present
twice in the data set, the (convex) models cannot discriminate between
them and might assign a weight on both columns.  (The situation is
worsened in the presence of more (near) duplicates, which
is typical in hyperspectral image data, for example). More recently,
Fu and Ma~\cite{FM16} improved the robustness analysis of the model
for $q = +\infty$ (in the absence of duplicates).

Another sparse regression model proposed in~\cite{BRRT12} and later improved 
in~\cite{GL13} is the following:
\begin{equation}
\label{GLmod}
    \min_{X \in \Rnn^{n,n}} \; \trace(X) \quad \text{s.t.} \quad
    \begin{aligned}
        & \norm{M - MX} \leq \epsilon\\
        & X(i,j) \leq X(i,i) \leq 1 \; \forall i,j.
    \end{aligned}
\end{equation}
(The model can easily be generalized for non-normalized $M$; see
model~\eqref{GLmod2}).
Here sparsity is enforced by minimizing the $\ell_1$ norm of
the diagonal of $X$ as $\trace(X) = \norm{\diag(X)}_1$ for $X \geq 0$,
while no off-diagonal entry of $X$ can be larger than the
diagonal entry in its row.  Hence $\diag(X)$ is sparse if and only
if $X$ is row sparse.

The model~\eqref{GLmod} is, to the best of our knowledge, the provably
most robust for near-separable NMF~\cite{GL13}. 
In particular, as opposed to most near-separable NMF algorithms that
require $M(:,\mathcal{K})$ to be full column rank, it only requires
the necessary condition that no column of $M(:,\mathcal{K})$ is
contained in the convex hull of the other columns. More precisely, let
us define the \emph{conical robustness} of a matrix $W \in \R^{m,r}$ as
\begin{equation*}
    \kappa = \min_{1 \leq k \leq r} \min_{x \in \Rnn^{r-1}}
        \norm{W(:, k) - W(:, \{1,\dotsc,r\}\setminus \{k\})x}_1.
\end{equation*}
We then say that $W$ is \emph{$\kappa$-robustly conical}, and the
following recovery result can be obtained:
\begin{theorem}[\cite{GL13}, Th.~7]
Let $M = M(:,\mathcal{K}) H$ be a separable matrix with
$M(:,\mathcal{K})$ being $\kappa$-robustly conical and the entries of
each column of $H$ summing to at most one, and let $\tilde{M} = M+N$.
 If $\epsilon \defby \max_{1 \leq j \leq
n} \norm{N(:,j)}_1 \leq \mathcal{O}\left(\frac{\kappa}{r}\right)$, then
the model~\eqref{GLmod} allows to recover the columns of
$M(:,\mathcal{K})$ up to error $\mathcal{O}\left(r
\frac{\epsilon}{\kappa}\right)$.
\end{theorem}

\subsection{Contribution and Outline of the Paper}

In the work~\cite{GL13} a robustness analysis of the
model~\eqref{GLmod} was given, which we here relate to the robustness
of the model~\eqref{eq:esser}.  We also present a practical and
efficient first order optimization method for~\eqref{GLmod}
(in~\cite{GL13} no such method was given).  More precisely, our
contribution in this work is threefold:

\begin{itemize}

\item In section~\ref{equival}, we prove that both sparse regression
models~\eqref{eq:esser} and~\eqref{GLmod} are equivalent. 
This significantly improves the theoretical
robustness analysis of~\eqref{eq:esser}, as the results
for~\eqref{GLmod} directly apply to~\eqref{eq:esser}.

\item In section~\ref{convFG}, we introduce a new model, very similar
to~\eqref{GLmod} (using the Frobenius norm, and not assuming
normalization of the input data), for which we propose an optimal
first-order method: the key contribution is a very efficient and
non-trivial projection onto the feasible set.  Although our approach
still requires $\mathcal{O}(mn^2)$ operations per iteration, it can
solve larger instances of~\eqref{GLmod} than commercial solvers, with
$n \sim 1000$. We show the effectiveness of our approach on synthetic
data sets in section~\ref{sec:synexp}.

\item In section~\ref{sec:hsiexp}, we preselect a subset of the columns of
the input matrix and scale them appropriately (depending on their
importance in the data set) so that we can apply our method
meaningfully to real-world hyperspectral images when $n \sim 10^6$.
We show that our approach outperforms state-of-the-art pure pixel
search algorithms.

\end{itemize}

\section{Equivalence between Sparse Regression Models \eqref{eq:esser}
and \eqref{GLmod}} \label{equival}


We now prove the equivalence between the models~\eqref{eq:esser} and~\eqref{GLmod}. 
We believe it is an important result because, as far as we know, both models have 
been treated completely independently in the literature, and, as explained in the Introduction,  
while model~\eqref{eq:esser} is more popular~\cite{EMO12,ESV12,FM16}, 
stronger theoretical guarantees were provided for model~\eqref{GLmod}~\cite{GL13}. 

\begin{theorem}
    \label{thm:equiv}
Let $\norm{\cdot}$ be a column wise matrix norm, that is, $\norm{A} =
\sum_{i} \alpha_i \norm{A(:,i)}_c$ for some $\alpha_i > 0$ and
some vector norm $\norm{\cdot}_c$.  Then~\eqref{eq:esser} is
equivalent to~\eqref{GLmod} in the following sense:
\begin{itemize}
    \item At optimality, both objective functions coincide,
    \item any optimal solution of~\eqref{GLmod} is an optimal solution
        of~\eqref{eq:esser}, and
    \item any optimal solution of~\eqref{eq:esser} can be trivially
        transformed into an optimal solution of~\eqref{GLmod}.
\end{itemize}
\end{theorem}
\begin{proof}
    See Appendix~\ref{sec:proof_equiv}.
\end{proof}

Theorem~2 implies that any robustness result for (2) applies to (3),
and vice versa.  It is therefore meaningful to compare the robustness
results of~\cite{GL13} and~\cite{FM16}.  It turns out that the results
in~\cite{GL13} are stronger because, as opposed to Fu and
Ma~\cite{FM16}, it does not require the absence of duplicated columns
(which is a rather strong assumption); see Theorem 1.  However, it is
interesting to note  that, in the absence of duplicated columns, both
robustness results essentially coincide (the error bounds are the same
up to some constant multiplicative factors), namely~\cite[Th.~2]{GL13}
and \cite[Th.~1]{FM16}.  We provide this result here for completeness: 

\begin{theorem}[\cite{GL13}, Th.~2] 
Let $M = M(:,\mathcal{K}) H$ where $M(:,\mathcal{K})$ is
$\kappa$-robustly conical and where the entries of each column of $H$
are at most one.  Let also $\tilde{M} = M+N$, and $H(i,j) \leq \beta <
1$ for all $1 \leq i \leq m$ and $j \notin \mathcal{K}$ (this is the
condition that there is no duplicates of the columns of
$M(:,\mathcal{K})$). If $\epsilon := \max_{1 \leq j \leq n}
\norm{N(:,j)}_1 < \frac{\kappa (1-\beta)}{20}$, then the
model~\eqref{GLmod} allows to recover the columns of
$M(:,\mathcal{K})$ up to error $\epsilon$.
\end{theorem} 

The advantage of the formulation \eqref{GLmod} over \eqref{eq:esser} is
that the objective function is smooth. Moreover, as we will show in
Section~\ref{sec:proj}, projecting onto the feasible set can be made
efficiently (even when the model is generalized to the case where the
columns of the input matrix are not normalized). Hence, we will be
able to apply an optimal first-order method of smooth convex
optimization. 

\section{Convex Model without Normalization and Fast Gradient Method}  \label{convFG}

In this section, in order to be able to obtain a more practical model 
that can be optimized using techniques from smooth convex optimization,  
we derive a new model, namely~\eqref{GLqp2}, closely related to~\eqref{GLmod} but where 
\begin{itemize}

\item the assumption $X \leq 1$ is not necessary, 

\item $||.||$ is the Frobenius norm, which is smooth and arguably the most popular choice in practice, and 

\item Lagrangian duality is used to incorporate the error term $||M-MX||_F^2$ in the objective function. 

\end{itemize} 
Then, we apply a fast gradient method on~\eqref{GLqp2} (Algorithm~\ref{alg:fgnsr}), 
after having listed related algorithmic approaches to tackle similar optimization problems.

\subsection{ Avoiding Column Normalization }

	The model~\eqref{GLmod} can be generalized in case $X \nleq 1$, 
	where $M = MX$, without column normalization of $M$. 
	The advantage is twofold: column normalization
	(i)~is only possible for nonnegative input matrix, and
	(ii)~may introduce distortion in the data set as it would be equivalent to consider that the noise added to each data point (that is, each column of $M$) is proportional to it~\cite{KSK12}. 
	If one wants to consider absolute error (the norm of each column of the noise is independent on the norm of the input data),  then	the input matrix should not be normalized and the following model should be considered~\cite{GL13}:
	\begin{align}
\min_{X \in \Omega} \;  \trace(X)  & \quad \text{ such that }  \quad  ||M - MX||_F \leq \epsilon . \label{GLmod2}
\end{align}
The set $\Omega$ is defined as
\begin{equation} \label{eqn:omega}
    \Omega \defby
    \{ X \in \Rnn^{n,n} \setcond
        X_{ii} \le 1,
        w_i X_{ij} \le w_j X_{ii} \forall i,j \},
\end{equation}
where the vector $w \in \Rnn^n$ are the column $\ell_1$ norms of $M$,
that is, $w_j = \norm{M(:,j)}_1$ for all $j$.  The upper bounds  $X_{ij}
\le \frac{w_j}{w_i} X_{ii}$ come from the fact that each weight used
to reconstruct a data point inside the convex cone generated by some
extreme rays cannot exceed the ratio of the $\ell_1$ norm of that data
point to each individual extreme ray.

Note that the optimization problem~\eqref{GLmod2} is convex, and it
can be solved as a second order conic program (SOCP) in $n^2$
variables.  This large number of variables even for moderate values of
$n$ rules out the use of off-the-shelf SOCP optimization software.  In
this section, we describe an optimal first-order method to solve
\eqref{GLmod2}. A main contribution is in the (non-trivial) projection
onto the feasible set $\Omega$.

\subsection{Related Work}

To solve a model similar to~\eqref{GLmod}, Bittorf et
al.~\cite{BRRT12} used a stochastic subgradient descent method, with a
non-smooth objective function (they were using the component-wise
$\ell_1$ norm of $M-MX$). Although the cost per iteration is
relatively low with $\mathcal{O}(n^2)$ operations per iterations, the
convergence is quite slow.

To solve~\eqref{eq:esser} with $q = +\infty$ in~\cite{EMO12} and $q =
2$ in~\cite{ESV12}, authors propose an alternating direction method of
multipliers (ADMM).  However, ADMM is not an optimal first-order
method as the objective function converges at rate $\mathcal{O}(1/k)$
vs.\@ $\mathcal{O}(1/k^2)$ for optimal first-order methods, where $k$
is the iteration number.  Moreover, the cost per iteration of ADMM is
larger as it requires the introduction of new variables (one variable
$Y$ of the same dimension as $X$, Lagrangian multipliers and a
parameter which is not always easy to tune).


Another optimal first-order method was proposed in~\cite{LA13}. However, it solves a rather different optimization problem, namely
\begin{equation*}
    \min_{X, Q} p^T \diag(X) + \beta \norm{MX - M - Q}_F^2 + \lambda \norm{Q}_1,
\end{equation*}
where two regularization parameters have to be tuned while the objective function is non-smooth, so that authors use a local linear approximation approach to smooth the objective function.
They also require column normalization while our approach does not.
Also, they point out that it would be good to incorporate the constraints from~\eqref{GLmod}, which we do in this paper by developing an effective projection on the feasible set.

\subsection{Fast Gradient Method for \eqref{GLmod2}}
\label{sec:FGNSR}

It is possible to solve \eqref{GLmod2} using commercial solvers that
are usually based on interior-point methods, such as Gurobi.  However,
it is computationally rather expensive as there are $\mathcal{O}(n^2)$
variables, where $n$ is the number of columns of $M$.

Moreover, it has to be noted that in the separable NMF case, it is not crucial to obtain high accuracy solutions: the main information one wants to obtain is which columns of $M$ are the important ones. Hence it is particularly meaningful in this context to use first-order methods (slower convergence but much lower computational cost per iteration). 

  The additional constraints that allows to take into account the fact that the columns of $M$ are not normalized makes the feasible set more complicated, but we develop an efficient projection method, that allows us to design an optimal first-order method (namely, a fast gradient method).

The problem we propose to solve is
\begin{equation} \label{GLqp2}
    \min_{X \in \Omega} F(X) = \frac{1}{2} \norm{M - MX}_F^2
        + \mu p^T \diag(X),
\end{equation}
where $M \in \R^{m,n}$ is the input data matrix, and $X \in \Omega$ are
the basis reconstruction coefficients. 
Note that we have replaced $\trace(X)$ with the more general term $p^T \diag(X)$ 
(they coincide if $p$ is the vector of all ones). 
The reason is twofold: 
(1)~it makes the model more general, and 
(2)~it was shown in~\cite{BRRT12} that using such a vector $p$ (e.g., randomly
chosen with its entries close to one) allows to discriminates between (approximate) duplicate basis
vectors present in the data. 
The penalty parameter $\mu \in \Rnn$ acts as a
Lagrange multiplier. From duality theory, 
there exists $\mu$ (which depends on the data $M$), 
such that models~\eqref{GLmod2} and~\eqref{GLqp2} are equivalent~\cite{wright1999numerical} (given that $p$ is the vector of all ones).

\algsetup{indent=2em}
\begin{algorithm}[ht!]
\caption{Fast Gradient Method for Nonnegative Sparse Regression with
    Self Dictionary (FGNSR)} \label{alg:fgnsr}
\begin{algorithmic}[1]
\REQUIRE A matrix $M \in \mathbb{R}^{m,n}$, number $r$ of columns to
    extract, a vector $p \in \mathbb{R}^n_{++}$ whose entries are close to 1, a
    penalty parameter $\mu$, and maximum number of iterations
    {\ttfamily maxiter}.

    \ENSURE An set $\mathcal{K} \subset \{1,\dotsc,n\}$ of column
    indices such that $\min_{H\in \Rnn^{r,n}} \norm{M - M(:,\mathcal{K}) H}_F$
    is small.

    \STATE \COMMENT{Initialization}
    \STATE $\alpha_0 \leftarrow 0.05$; $Y \leftarrow 0_{n,n}$;
        $X \leftarrow Y$; $L \leftarrow \sigma_{\max}(M)^2$;
        \label{line:init}
    \FOR{$k = 1 :$ {\ttfamily maxiter}}
        \STATE $Y_p \leftarrow Y$;
        \STATE $\nabla F(X) \leftarrow M^T M X - M^T M + \mu \diag(p)$;
            \label{line:GEMM}
        \STATE\COMMENT{Projection on $\Omega$; see Section~\ref{sec:proj}}
        \STATE $Y \leftarrow \mathcal{P}_{\Omega}\left(
            X - \frac{1}{L} \nabla F(X) \right)$;\label{line:proj}
        \STATE $X \leftarrow Y + \beta_k (Y - Y_p)$, \;
            where $\beta_k =  \frac{\alpha_{k-1} (1-\alpha_{k-1})}{
            \alpha_{k-1}^2 + \alpha_{k}}$ such that $\alpha_k \geq 0$ and
            $\alpha_k^2 = (1-\alpha_{k}) \alpha_{k-1}^2$;
    \ENDFOR
    \STATE $\mathcal{K} \leftarrow$ postprocess($X$,$r$);
    \COMMENT{The simplest way is to pick the $r$ largest entries of
    $\diag(X)$ as done in \cite{BRRT12}.  In the presence of
    (near-)duplicated columns of $M$, one should use more sophisticated
    strategies~\cite{GL13}.} \label{line:postprocess}
\end{algorithmic}
\end{algorithm}

Algorithm~\ref{alg:fgnsr} is a first-order method for minimizing
$F(X)$ over $\Omega$, based on Nesterov's fast gradient
method~\cite{Y04}.  Here ``fast'' refers to the fact that it attains
the best possible convergence rate of $\mathcal{O}(1/k^2)$ in the
first-order regime.  Because ${M}$ is not necessarily full rank (in
particular $\rank({M}) \leq r$ when $M$ is a $r$-separable matrix
without noise), the objective function of \eqref{GLqp2} is not
necessarily strongly convex. However, its gradient is Lipschitz
continuous with constant $L = \lambda_{\max}({M}^T {M}) =
\sigma_{\max}({M})^2$, which is sufficent to guarantee the claimed
convergence rate.

The requested number of columns $r$ in Algorithm~\ref{alg:fgnsr} is
used only in the postprocessing step (line~\ref{line:postprocess}).
Hence upon termination the obtained matrix $X$ can be used to
extract multiple NMFs, corresponding to different ranks, and to pick
the most appropriate one among them for the application at hand.


The penalty parameter $\mu$ in \eqref{GLqp2} is crucial as it balances the
importance between the approximation error $\norm{M - M X }_F^2$ and the fact that we want the diagonal of $X$ to be as sparse as possible.
On one hand, if $\mu$ is too large, then the term $\norm{M - M X}_F^2$ will not have much importance in the objective function
leading to a poor approximation. On the other hand, if $\mu$ is too
small, then $\norm{M - M X}_F^2$ will have to be very small and $X$
will be close to the identity matrix. However, in our experience, it seems
that the output of Algorithm~\ref{alg:fgnsr} is not too sensitive to this scaling.
The main a reason is that only the largest entries of (the diagonal of) $X$ will be extracted by the post-processing procedure while the most representative columns of $M$ remains the same independently of the value of $\mu$.
In other words, increasing $\mu$ will have the effect of increasing in
average the entries of $X$ but the rows corresponding to the important
columns of $M$ will continue having larger entries. Therefore the extracted index set $\mathcal{K}$ will remain the same.

To set the value of $\mu$, we propose the following heuristic which appears to work very well in practice:
\begin{itemize}

\item Extract a subset $\mathcal{K}$ of $r$ columns of ${M}$ with the fast algorithm proposed in \cite{GV14} (other fast separable NMF algorithms would also be possible);

\item Compute the corresponding optimal weight $H$,
\begin{equation*}
    H = \argmin_{Z \in \Rnn^{r,n}} || {M} - {M}(:,\mathcal{K}) Z ||_F^2,
\end{equation*}
using a few iterations of coordinate descent; see \cite{GG12}.

\item Define $X_0(\mathcal{K},:) = H$ and $X_0(i,:) = 0$ for all $i
    \notin \mathcal{K}$.

\item Set $\mu = \frac{ \norm{{M} - {M} X_0}_F^2 }{p^T \diag(X_0)}$,
    to balance the importance of both terms in the objective function.

\end{itemize}

Note that if the noise level $\epsilon$, or an estimate thereof, is
given as an input, $\mu$ can be easily updated in the course of the
gradient iteration so that $\norm{M-MX}_F \approx \epsilon$: If
$\norm{M-MX}_F$ is too small (large) relative to $\epsilon$ in the
course of the gradient iteration, $\mu$ is simply increased
(decreased), and the method is restarted.  Of course these adjustments
should be carried out in a convergent scheme, say, geometrically
decreasing, in order to maintain convergence of Algorithm~\ref{alg:fgnsr}.

\begin{remark}
Algorithm~\ref{alg:fgnsr} can be directly generalized to any other
smooth norm for which the gradient is Lipschitz
continuous and can be computed efficiently. 
\end{remark}

\subsection{Euclidean Projection on $\Omega$}

\label{sec:proj}

In Algorithm~\ref{alg:fgnsr} we need to compute the Euclidean projection of a
point $X \in \R^{n,n}$ on the set $\Omega$ from Equation~\eqref{eqn:omega}, denoted $\mathcal{P}_{\Omega}$.
Recall that for a convex subset $C \in
\R^n$ of an Euclidean vector space, a function $\phi: \R^n \rightarrow
C$ is an Euclidean projection on $C$ if for all $x \in \R^n$
\begin{equation*}
    \norm{x - \phi(x)} = \min_{z \in C} \norm{x - z}.
\end{equation*}
We describe in Appendix~\ref{sec:projection} how to compute this projection
efficiently. More precisely, we show how to solve the problem $\min_{Z
\in \Omega} \norm{X - F}_F$ in $\mathcal{O}(n^2 \log n)$ operations.
In the unweighted case, that is, $w_j \equiv 1$ for all $1 \le j \le
n$, our algorithm is similar to the one described in~\cite{BRRT12},
but the inclusion of non-unit weights makes the details very much
different.  The worst case bound of $\mathcal{O}(n^2 \log n)$
operations will typically overestimate the computational cost if
appropriate data structures are used.  This is explained in
Remark~\ref{rem:datastruct}, Appendix~\ref{sec:projection}.

\subsection{Computational Cost}

In order to find the asymptotic computational cost of
Algorithm~\ref{alg:fgnsr}, we analyze the three main
steps as follows.
\begin{itemize}

    \item[Line \ref{line:init}:] The maximum singular value of an
        $m$-by-$n$ matrix can be well approximated with a few steps of
        the power method, requiring $\mathcal{O}\left( mn \right)$
        operations.

    \item[Line \ref{line:GEMM}:] The matrix ${M}^T {M}$ should be
        computed only once at a cost of $\mathcal{O}\left( m n^2
        \right)$ operations. If $m \geq 2n$, then computing $({M}^T
        {M}) X$ requires $\mathcal{O}\left( n^3 \right)$ operations.
        Otherwise, one should first compute ${M}X$ at a cost of
        $\mathcal{O}\left( m n^2 \right)$ operations and then ${M}^T
        ({M} X)$ at a cost of $\mathcal{O}\left( m n^2 \right)$
        operations (the total being smaller than $n^3$ if $m \leq
        2n$).

    \item[Line \ref{line:proj}:] The projection onto $\Omega$ of an
        $n$-by-$n$ matrix $X$ requires $\mathcal{O}\left( n^2  \log
        n\right)$ operations (the $\log n$ factor comes from the fact
        that we need to sort the entries of each row of $X$); see
        Section~\ref{sec:proj} for the details about the projection
        step.  Note that each row of $X$ can be projected
        independently hence this step is easily parallelizable.
        Moreover, many rows of $X$ are expected to be all-zeros and
        their projection is trivial.

\end{itemize}

Hence the number of operations is in $\mathcal{O}(mn^2 + n^2 \log
n)$, and since we typically have $m \geq \log n$,
this reduces to $\mathcal{O}\left(m n^2 \right)$ operations.

The computational cost could potentially be decreased using random
projections of the data points to reduce the dimension $m$ of the
input matrix; see, e.g.,~\cite{ding2013topic, benson2014scalable}.  It
would be interesting to combine these techniques with
Algorithm~\ref{alg:fgnsr} in future work.

\section{Numerical Experiments}
\label{sec:exp}

We now study the noise robustness of Algorithm~\ref{alg:fgnsr}
numerically, and compare it to several other state-of-the-art methods
for near-separable NMF problems on a number of hyperspectral image
data sets.  We briefly summarize the different algorithms under
consideration as follows.

\begin{asparadesc}

    \item[Successive projection algorithm (SPA).] SPA extracts
        recursively $r$ columns of the input matrix ${M}$.  At each
        step, it selects the column with the largest $\ell_2$ norm,
        and projects all the columns of ${M}$ on the orthogonal
        complement of the extracted column~\cite{MC01}.  SPA was shown
        to be robust to noise~\cite{GV14}. SPA can also be interpreted
        as a greedy method to solve the sparse regression model with
        self dictionary~\cite{FM15}.

    \item [XRAY.] It recursively extracts columns of the input
        unnormalized matrix ${M}$ corresponding to an extreme ray of
        the cone generated by the columns of ${M}$, and then projects
        all the columns of ${M}$ on the cone generated by the
        extracted columns.  We used the variant referred to as
        ``max''~\cite{KSK12}.

    \item[Successive nonnegative projection algorithm (SNPA).] A
        variant of SPA using the nonnegativity constraints in the
        projection step~\cite{G14b}.  To the best of our knowledge, it
        is the provably most robust sequential algorithm for separable
        NMF (in particular, it does not need $M(:\mathcal{K})$ to be
        full rank).

    \item[Exact SOCP solution.]  We solve the exact
        model~\eqref{GLmod2} using the SOCP solver of
        Gurobi\footnote{\url{https://www.gurobi.com}}, an interior
        point method.  The obtained solution
        will serve as a ``reference solution''.

    \item[FGNSR.]  A {\textsc Matlab}/C implementation of
        Algorithm~\ref{alg:fgnsr}, which is publicly
        available\footnote{\url{https://github.com/rluce/FGNSR}}.

\end{asparadesc}

Our comparison does not include algorithms using linear functions to
identify vertices (such as the pure pixel index algorithm~\cite{B94}
and vertex component analysis~\cite{ND05}) because they are not robust
to noise and do not perform well for the challenging synthetic data
sets described hereafter; see~\cite{GV14}.

In all our experiments with FGNSR and the exact SOCP solution
to~\eqref{GLmod2} we use the simplest postprocessing to extract the
sought for index set $\mathcal{K}$ from the solution matrix $X$
in~\eqref{GLmod2}: We always pick the indices of the $r$ largest
diagonal values of $X$ (see final step in Algorithm~\ref{alg:fgnsr}).

Table~\ref{comptable} summarizes the following information for the
different algorithms: computational cost, memory requirement,
parameters, and whether the $H$ is required to be column normalized.
The FLOP count and memory requirement for the exact solution of the
SOCP via an interior point method depends on the actual SOCP
formulation used, as well as on the sparsity of the resulting problem.
In any case, they are orders of magnitudes greater than for the other
algorithms.

We complement these information with the
average wall clock run times for the small ``middlepoint'' matrices
from Section~\ref{sec:synexp} ($m=50$, $n=55$) in the last column.
More run time results are shown in Section~\ref{sec:hsiexp}.

\begin{table*}[ht]
\begin{center}
\caption{Complexity Comparison for a $m$-by-$n$ Input Matrix.
    \label{comptable}}
\begin{tabular}{|c||c|c|c|c|c|}
\hline
& FLOPs    & Memory    &  Parameters &  Normalization  & Run time in
sec.~\ref{sec:synexp}\\  \hline  \hline
SPA  & $2mnr$ + $\mathcal{O}(m r^2)$ &	   $\mathcal{O}(mn)$  & $r$ & Yes & $<0.01$s\\
XRAY & $\mathcal{O}(mnr)$    & $\mathcal{O}(mn)$  & $r$ & No & $0.03$s\\
SNPA & $\mathcal{O}(mnr)$ &	   $\mathcal{O}(mn)$  & $r$ & Yes & $0.10$s\\
SOCP (Gurobi, IPM) & $\mathcal{O}(n^6)$  & $\mathcal{O}(n^4)$ & $\norm{N}_F$ & No & $2.78$s\\
FGNSR & $\mathcal{O}(m n^2)$ &  $\mathcal{O}(mn + n^2)$ &  $r$ or $\norm{N}_F$ & No   & $0.09$s\\ \hline
\end{tabular}
\end{center}
\end{table*}

In the following section~\ref{sec:synexp} we study numerically the noise
robustness of the model~\ref{GLqp2} on an artificial dataset, and in
section~\ref{sec:hsiexp} we compare the methods from above to
real-world hyperspectral image data sets.

\subsection{Robustness Study on Synthetic Datasets}
\label{sec:synexp}

The data set we consider is specifically designed to test algorithms
for their robustness against noise.  We set $m = 50$,
$n = 55$ and $r = 10$. Given the noise level $\epsilon$, a noisy $r$-separable matrix
\begin{equation}
    \label{eq:middata}
    M = WH + N \in \R^{m,n}
\end{equation}
is generated as follows:

\begin{itemize}

    \item Each entry of the matrix $W$ is generated uniformly at
        random in the interval $[0,1]$ (using the {\ttfamily rand}
        function of \textsc{Matlab}), and each column of $W$ is then
        normalized so that it sums to one.

    \item The first $r$ columns of $H$ are always taken as the
        identity matrix to satisfy the separability assumption. The
        remaining $\frac{r(r-1)}{2} = 45$ columns of $H$ contain all
        possible combinations of two nonzero entries equal to 0.5 at
        different positions.  Geometrically, this means that these
        $45$ columns of $M$ are the \emph{middle points} of all the
        pairs from the columns of $W$.

    \item No noise is added to the first $r$ columns of $M$, that is,
        $N(:,j) = 0$ for all $1 \leq j \leq r$, while all the other
        columns corresponding to the middle points are moved towards
        the exterior of the convex hull of the columns of $W$.
        Specifically, we set
        \begin{equation*}
            N(:,j) = M(:,j)-\bar{w},  \quad \text{ for } r+1 \leq j \leq n,
        \end{equation*}
        where $\bar{w}$ is the average of the columns of $W$
        (geometrically, this is the vertex centroid of the convex hull
        of the columns of $W$).  Finally, the noise matrix $N$ is
        scaled so that it matches the given noise level $\norm{N}_F =
        \epsilon$.

\end{itemize}
Finally, in order to prevent an artificial bias due to the ordering
in which $H$ is constructed, the columns of $M$ are randomly permuted.
We give the following illustration of this type of data set (with $m=r=3$):
\begin{center}
\includegraphics[width=5cm]{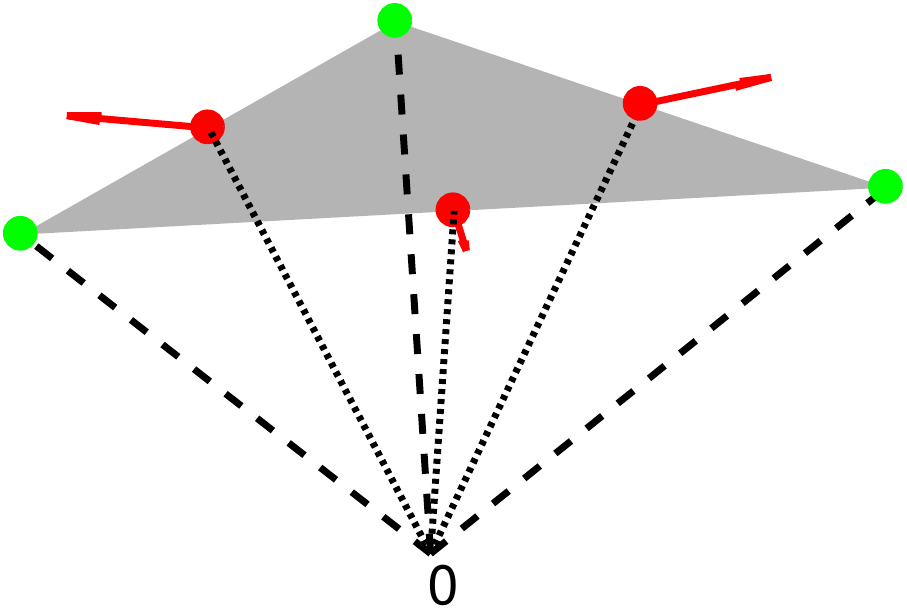}
\end{center}
The shaded area shows the convex hull of $W$, and the arrow attached
to the middle points indicate the direction of the noise added to
them.  With increasing noise level $\epsilon$, any algorithm for
recovering the conic basis $W$ will eventually be forced to select
some displaced middle points, and hence will fail to identify $W$,
which makes this data set useful for studying the noise robustness of
such algorithms.

In order to compare the algorithms listed at the beginning of the
section, two measures between zero and one will be used, one being the
best possible value and zero the worst: given a set of indices
$\mathcal{K}$ extracted by an algorithm, the measures are as follows:
\begin{itemize}

    \item \emph{MRSA.} We compute the mean-removed spectral angle
        between a selected basis column $w$ and the true basis
        column $w_*$, according to
        \begin{equation*}
            \arccos\left(\frac{\langle w - \bar{w}, w_* -
                \bar{w}_* \rangle}{\norm{w - \bar{w}} \norm{w_* -
        \bar{w}_*}}\right),
        \end{equation*}
        and normalizing the result to the interval $[0,100]$ (a value
        of zero is a perfect match).  In order to obtain a single
        number for a given computed basis matrix $W$ we take the mean
        of all individual MRSAs.

    \item \emph{Relative approximation error.} It is defined as
        \begin{equation*}
            1.0 - \frac{\min_{H \geq 0}
                \norm{M-M(:,\mathcal{K})H}_F}{\norm{M}_F} .
        \end{equation*}
        (Taking $H = 0$ gives a measure of zero).

\end{itemize}

\begin{figure*}[t]
\begin{center}
\includegraphics[width=0.48\textwidth]{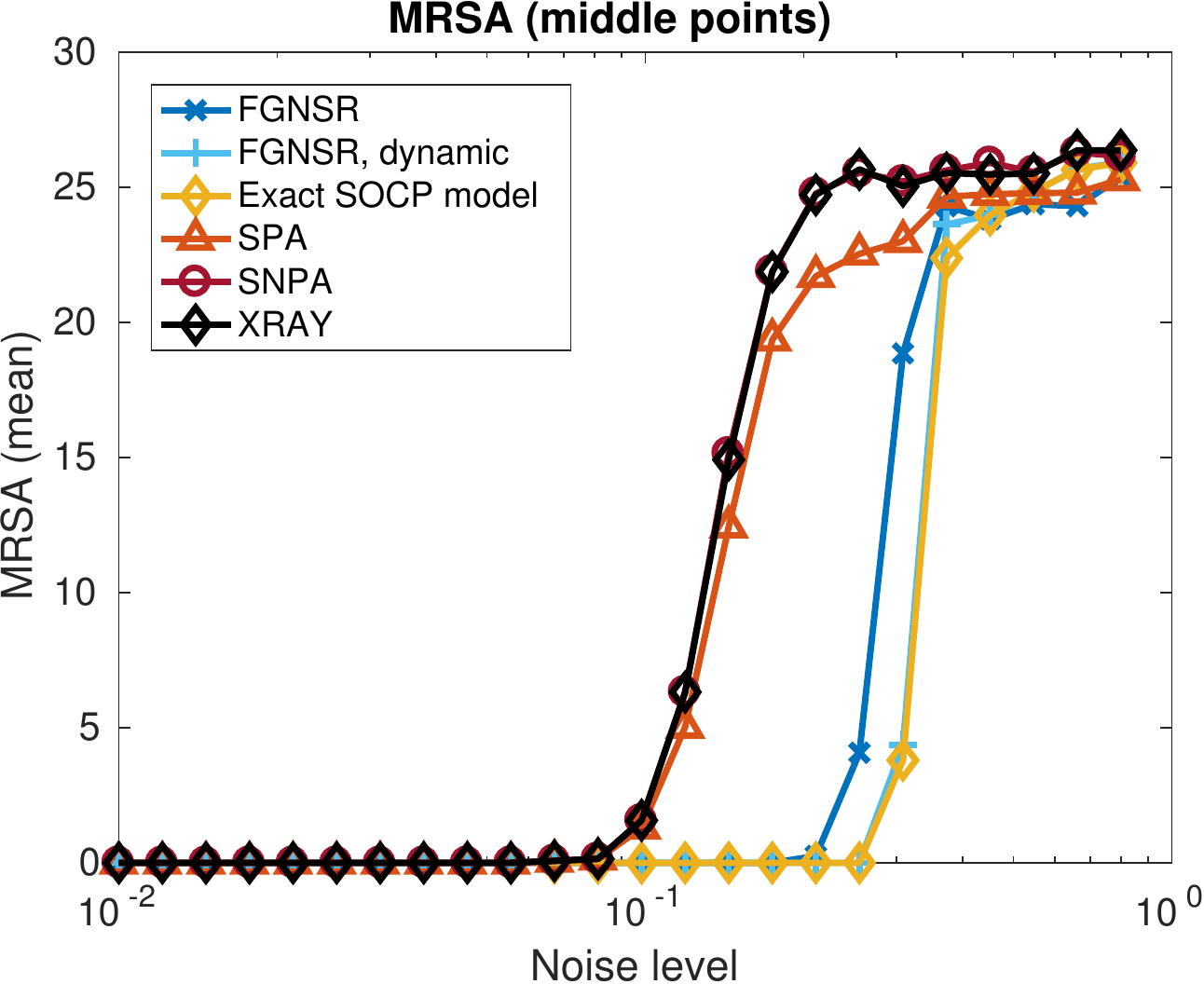}
\hfill
\includegraphics[width=0.48\textwidth]{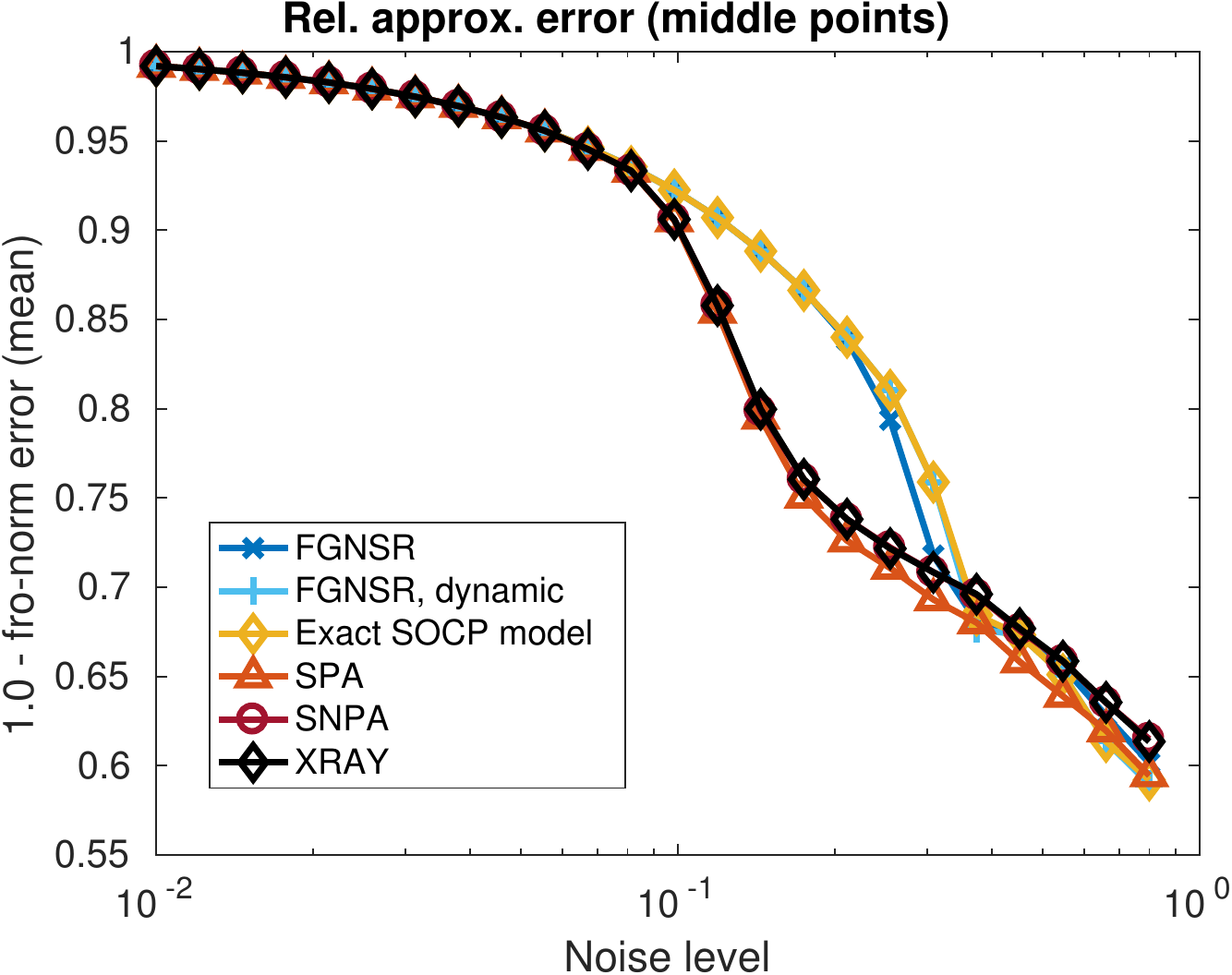}
\caption{Robustness study of various near-separable algorithms on the
    middle point set (see Sec.~\ref{sec:synexp}).  {\itshape Left:}
    Index recovery measure.  Note that the results of ``XRAY'' and
    ``SNPA'' are visually almost indistinguishable, as are those of
    ``Exact SOCP model'' and ``FGNSR, dynamic''. {\itshape Right:}
    Relative approximation error.\label{fig:midexp}}
\end{center}
\end{figure*}

Figure~\ref{fig:midexp} shows these two measures over a series of data
sets over increasing noise level $\epsilon$.   For each noise level, a
random middle point data set as described above was generated $25$
times, and the average of the respective measure over this sample
yields one data point at noise level $\epsilon$.

Algorithm~\ref{alg:fgnsr} is clearly superior to all other algorithms,
and recovers the true conic basis even at quite large noise levels.
With the heuristic choice for the multiplier $\mu$ (see
Section~\ref{sec:FGNSR}), the robustness is still slightly inferior to
the exact SOCP solution.  The results labelled ``FGNSR, dynamic''
refers to a variant of Algorithm~\ref{alg:fgnsr} where $\mu$ is
heuristically adjusted in the course of the gradient iteration so that
$\norm{M - MX}_F \approx \epsilon$. (Similarly, one could steer $\mu$
towards a prescribed value of $\trace(X)$).

Note that by construction the $\ell_1$ norm of all the columns in
$H$ in~\eqref{eq:middata} is $1.0$, which is in fact a requirement by
for some the algorithms considered here (see Table~\ref{comptable}).
It is an important feature of Algorithm~\ref{alg:fgnsr} that it is
also applicable if the columns of $H$ are not normalized.  We now
study this case in more detail.

Consider the following slight variation of the middle point data from
above:  Instead of placing the middle points by means of a convex
combination of two vertices, we now allow for conic combination of
these pairs, i.e., the middle point will be randomly scaled by some
scalar in $[\alpha^{-1}, \alpha]$ (we take $\alpha=4$).  The
following picture illustrates these \emph{scaled middle point data}:

\begin{center}
\includegraphics[width=5cm]{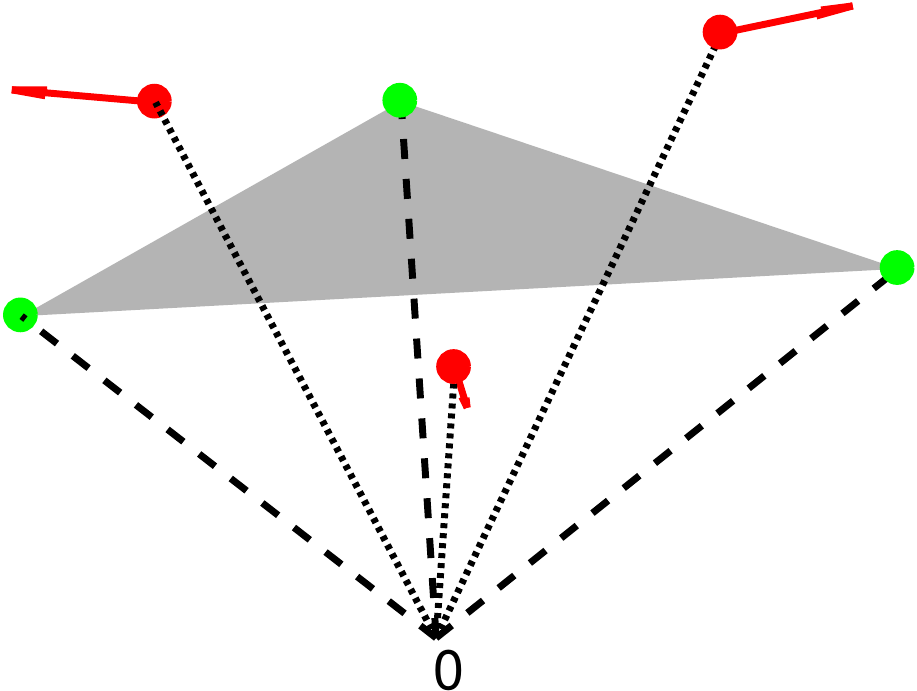}
\end{center}

We compare the algorithms listed at the beginning of this section on
this data set exactly as described above.  the results are shown in
Figure~\ref{fig:midexp_scaled}.  The results labeled ``normalize,
SPA'' and ``normalize, FGNSR'' refer to $\ell_1$-normalizing the
columns of the input matrix $M$ prior to applying SPA and FGNSR,
respectively.   From the results it is clear that FGNSR is by far the
most robust algorithm in this setting.

\begin{figure*}[t]
\begin{center}
\includegraphics[width=0.48\textwidth]{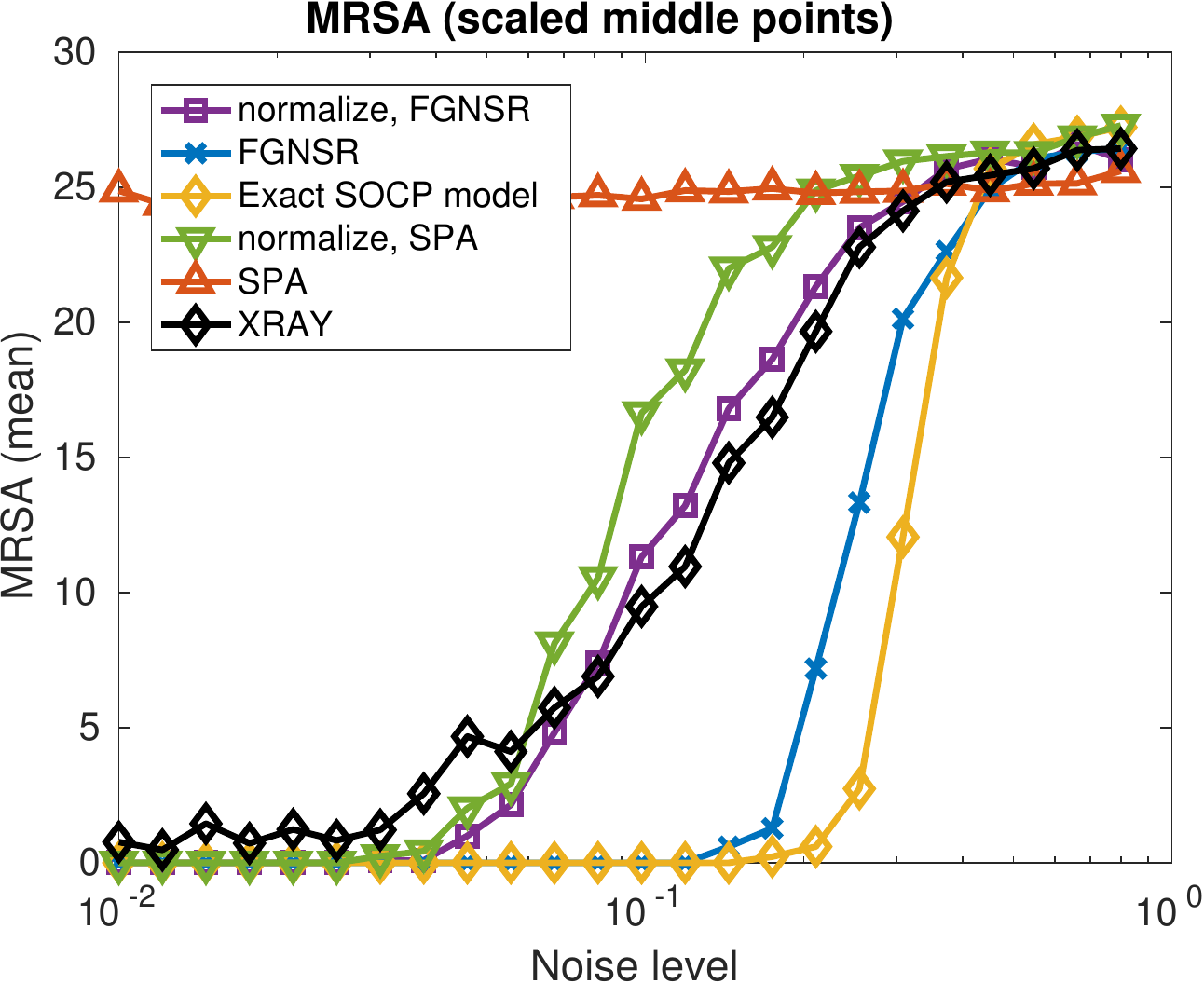}
\hfill
\includegraphics[width=0.48\textwidth]{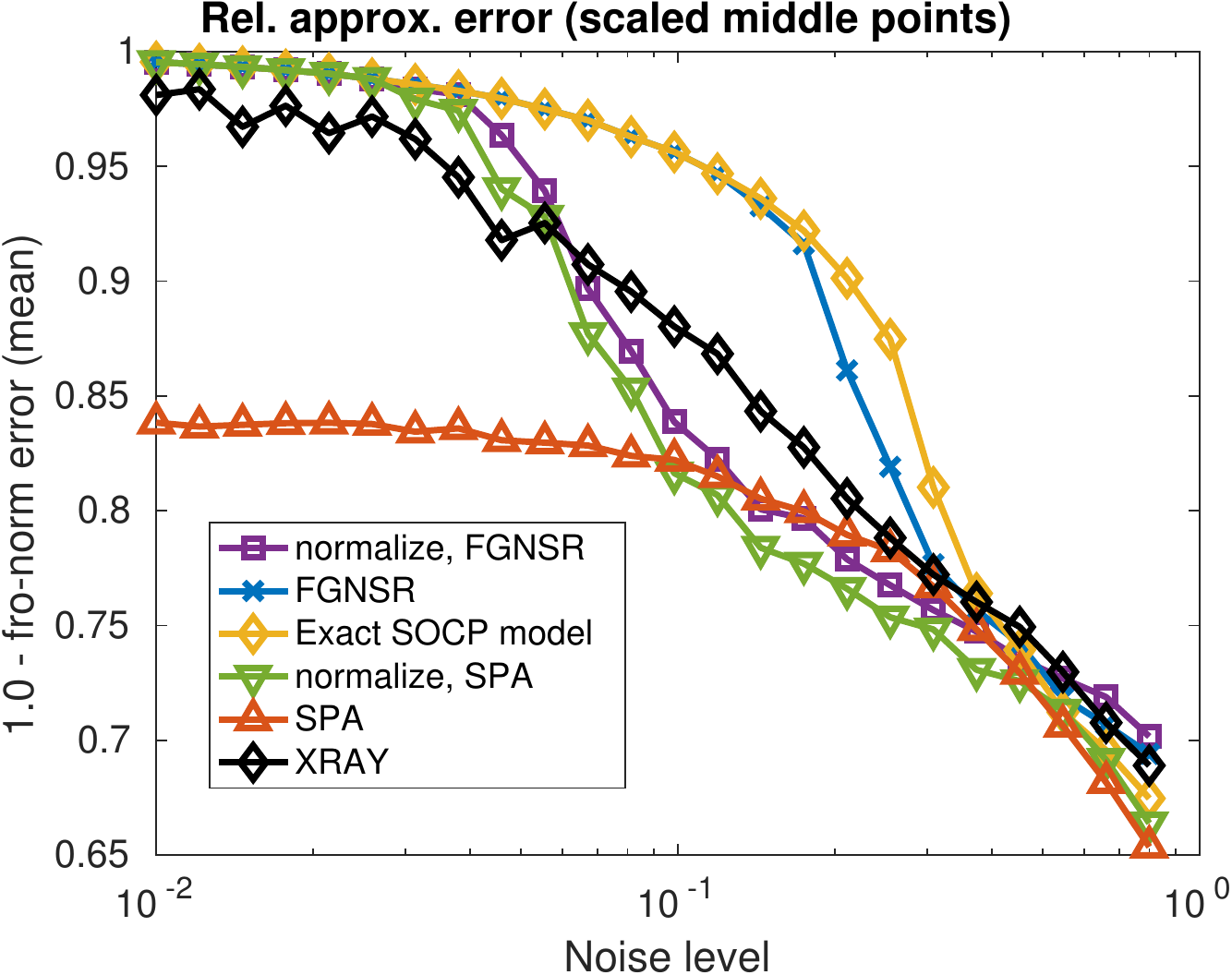}
\caption{Robustness study of various near-separable algorithms on the
    scaled middle point set (see
    Sec.~\ref{sec:synexp}).  The show data is analogous to the data in
    Figure~\ref{fig:midexp}.  The results for SNPA are not shown here
    to allow for a cleaner presentation; they are very similar to the ones for SPA. \label{fig:midexp_scaled}}
\end{center}
\end{figure*}

\subsection{Blind Hyperspectral Unmixing}

\label{sec:hsiexp}

A hyperspectral image (HSI) measures the fraction of light reflected
(the reflectance) by the pixels at many different wavelengths, usually
between 100 and 200. For example, most airborne hyperspectral systems
measure reflectance for wavelengths between 400nm and 2500nm, while
regular RGB images contain the reflectance for three visible
wavelengths: red at 650nm, green at 550nm and blue at 450nm.  Hence,
HSI provide much more detailed images with information invisible to
our naked eyes.  A HSI can be represented as a nonnegative $m$-by-$n$
matrix where the entry $(i,j)$ of matrix $M$ is the reflectance of the
$j$th pixel at the $i$th wavelength, so that each column of $M$ is the
so-called spectral signature of a given pixel.  Assuming the linear
mixing model, the spectral signature of each pixel equals the linear
combination of the spectral signatures of the constitutive materials
it contains, referred to as endmembers, where the weights correspond
to the abundance of each endmember in that pixel. This is a simple but
natural model widely used in the literature. For example, if a pixel
contains 60\% of grass and 40\% of water, its spectral signature will
be 0.6 times the spectral signature of the grass plus 0.4 times the
spectral signature of water, as 60\% is reflected by the grass and
40\% by the water.  Therefore, we have \[ M(:,j) = \sum_{k=1}^r W(:,k)
H(k,j) + N(:,j) , \] where $M(:,j)$ is the spectral signature of the
$j$th pixel, $r$ is the number of endmembers, $W(:,k)$ is the spectral
signature of the $k$th endmember, $H(k,j)$ is the abundance of the
$k$th endmember in the $j$th pixel, and $N$ represents the noise (and
modeling errors).  In this context, the separability assumption is
equivalent to the so-called pure-pixel assumption that requires that
for each endmember there exists a pixel containing only that
endmember, that is, for all $k$, there exists $j$ such that $M(:,j)
\approx W(:,k)$.

The theoretical robustness results of near-separable NMF algorithms do
not apply in most cases, the reasons being that
\begin{inparaenum}[(i)]
    \item the noise level is usually rather high,
    \item images contain outliers,
    \item the linear mixing model itself is incorrect (in particular
        because of multiple interactions of the light with the
        surface, or because of its interaction with the atmosphere),
    \item  the pure-pixel assumption is only approximately satisfied
        (or only some endmembers have pure pixels), and
    \item the number of endmembers is unknown (and usually endmembers
        in small proportion are considered as noise);
        see~\cite{Jose12} and the references therein.
\end{inparaenum}

However, near-separable NMF algorithms (a.k.a.\@ pure-pixel search
algorithms) usually allow to extract pure (or almost pure) pixels and
are very popular in the community; for example NFIND-R~\cite{Win99} or
vertex component analysis (VCA)~\cite{ND05}. They can also be
particularly useful to initialize more sophisticated method not based
on the pure-pixel assumption; see~\cite{Ma14}.


For most HSI, $n$ is of the order of millions, and it is impractical
to solve~\eqref{GLqp2} with either our fast gradient method or even
the interior point solver of Gurobi, as we did in
Section~\ref{sec:synexp}.
%
In the next section, we adopt a strategy similar to that
in~\cite{EMO12} where a subset of pixels is first selected as a
preprocessing step. Then, we apply our model \eqref{GLqp2} on that
subset using an appropriate strategy which is described in the next
section.


\subsubsection{Subsampling and Scaling of HSI's}

A natural way to handle the situation when $n$ is large is to
preselect a subset of the columns of ${M}$ that are representative of
the data set. In~\cite{EMO12}, authors used $k$-means to identify that
subset.  However, $k$-means has several drawbacks: it cannot handle
different scaling of the data points, and scales badly as the number
of clusters increases, running in $\mathcal{O}(mnC)$ where $C$ is the
number of clusters to generate.  A much better alternative, that was
specifically designed to deal with HSI, is the hierarchical clustering
procedure developed in~\cite{GDK14}.  The computational cost is
$\mathcal{O}(mn \log_2 C)$ (given that it generates well-balanced
clusters).

Keeping only the centroids generated by a clustering algorithm is a
natural way to subsample HSI.  However, it is important to take into
account the importance of each centroid, that is, the number of data
points attached to it.  Let $\mathcal{C}$ be the index set
corresponding to the centroids $M(:,\mathcal{C})$ of the extracted
clusters.  For $\card{\mathcal{C}}$ sufficiently large, each pixel will
be relatively close to its centroid: mathematically, for all $j$,
there exists $k \in \mathcal{C}$ such that $M(:,j) \approx M(:,k)$.
If we only allow the centroids in the dictionary and denote $X \in
\mathbb{R}^{\card{\mathcal{C}}, \card{\mathcal{C}}}$ the corresponding
weights, the error term can be approximated by
\begin{align*}
    & \norm{M - M(:,K) X}_F^2
    = \sum_{j=1}^n \norm{M(:,j) - M(:,K) X(:,j)}_F^2  \\
    & \approx \sum_{k \in \mathcal{C}} n_k \norm{M(:,k) - M(:,K) X(:,k)}_F^2 \\
    & = \sum_{k \in \mathcal{C}}  \norm{\sqrt{n_k} M(:,k) - \sqrt{n_k} M(:,K) X}_F^2,
\end{align*}
where $n_k$ the number of pixels in the $k$th cluster.

Therefore, in this section, we apply our model only to the matrix
$M(:,\mathcal{C})$ where each centroid is scaled according to the
square root of the number of points belonging to its cluster. This
allows us to take into account the importance of the different
clusters.  For example, an outlier will correspond to a cluster with a
single data points (provided that $\card{\mathcal{C}}$ sufficiently
large) hence its influence in the objective function will be
negligible in comparison with large clusters.

\paragraph*{Postprocessing of $X$}

In the synthetic data sets, we identified the subset $\mathcal{K}$ using the $r$ largest entries of
$X$. It worked well because (i)~the data sets did not contain any outlier,
and (ii)~there were no (near-)duplicated columns in the data sets.
In real data sets, these two conditions are usually not met.
Note however that the preprocessing clustering procedure aggregates (near-)duplicated columns. However, if a material is present in very large proportion of the image (e.g., the grass in the Urban data sets; see below), several clusters will be made mostly of that material.

Therefore, in order to extract
a set of column indices from the solution matrix $X$ (see
Algorithm~\ref{alg:fgnsr}, line~\ref{line:postprocess}), we will use a more sophisticated strategy.

The $i$th row of matrix $X$ provides the weights necessary to reconstruct each column of $M$ using the $i$th column of $M$ (since $M \approx MX$), while these entries are bounded by the diagonal entry $X_{ii}$.
From this, we note that
\begin{compactenum}[(i)]

\item If the $i$th row corresponds to an outlier, it will in general
    have its corresponding diagonal entry $X_{ii}$ non-zero but the
    other entries will be small (that is, $X_{ij}$ $j \neq i$).
    Therefore, it is important to also take into account off-diagonal
    entries of $X$ in the postprocessing: a row with a large norm will
    correspond to an endmember present in many pixels.  (A similar
    idea was already proposed in \cite[Section 3]{GV14}.)

\item Two rows of $X$ that are close to one another (up to a scaling
    factor) correspond to two endmembers that are present in the same
    pixels in the same proportions. Therefore, it is likely that these
    two rows correspond to the same endmember.  Since we would like to
    identify columns of $M$ that allow to reconstruct as many pixels
    as possible, we should try to identify rows of $X$ that are as
    different as possible. This will in particular allow us to avoid
    extracting near-duplicated columns.

\end{compactenum}

Finally, we need to identify rows (i)~with large norms (ii)~that are as different as one another as possible. This can be done using SPA on $X^T$: at each step, identify the row of $X$ with the largest norm and project the other rows on
its orthogonal complement (this is nothing but a QR-factorization with column pivoting).
We observe in practice that this postprocessing is particularly effective at avoiding outliers and near-duplicated columns (moreover, it is extremely fast).

\subsubsection{Experimental Setup}

In the following sections, we combine the hierarchical clustering procedure with our near-separable NMF algorithm and compare it with state-of-the-art pure-pixel search algorithms (namely SPA, VCA, SNPA, H2NMF and XRAY) on several HSI's.
We have included vertex component analysis (VCA)~\cite{ND05} because it is extremely popular in the hyperspectral unmixing community,
although it is not robust to noise~\cite{GV14}.
VCA is similar to SPA except that (i)~it first performs dimensionality reduction of the data using PCA to reduce the ambient space to dimension $r$, and
(ii)~selects the column maximizing a randomly generated linear function.

Because the clustering procedure already does some work to identify candidate pure pixels, it could be argued that the comparison between our hybrid approach and plain pure-pixel search algorithms is unfair. Therefore, we will also apply SPA, VCA, XRAY, H2NMF and SNPA on the subsampled data set. We subsample the data set by selecting 100 (resp.\@ 500) pixels using H2NMF, and denote the corresponding algorithms SPA-100 (resp.\@ SPA-500), VCA-100 (resp.\@ VCA-500), etc.

Because it is difficult to assess the quality of a solution on a real-world HSI, we use the relative error in percent: given the index set $\mathcal{K}$ extracted by an algorithm, we report
\begin{equation*}
    100  \frac{\min_{H \geq 0} \norm{M - M(:,\mathcal{K}) H}_F}{\norm{M}_F},
\end{equation*}
where $M$ is always the full data set.

The \textsc{Matlab} code used in this study is
available\footnote{\url{https://sites.google.com/site/nicolasgillis/}},
and all computations were carried out with \textsc{Matlab}-R2015b on a
standard Linux/Intel box.

\subsubsection{Data Sets and Results}

We will compare the different algorithms on the following data sets:
\begin{itemize}

    \item The Urban
        HSI\footnote{\url{http://www.erdc.usace.army.mil/}} is taken
        from  HYper-spectral Digital Imagery Collection Experiment
        (HYDICE) air-borne sensors, and contains 162 clean spectral
        bands where each image has dimension $307 \times 307$.  The
        corresponding near-separable nonnegative data matrix $M$
        therefore has dimension $162$ by $94249$.  The Urban data is
        mainly composed  of 6 types of materials: road, dirt, trees,
        roofs, grass and metal (as reported in~\cite{GWO09}).

    \item The San Diego airport HSI is also from the HYDICE air-borne
        sensors. It contains 158 clean bands, with $400 \times 400$
        pixels for each spectral image hence $M \in \mathbb{R}^{160000
        \times 158}_+$. There are about eight types of materials:
        three road surfaces, two roof tops, trees, grass and dirt;
        see, e.g.,~\cite{GDK14}.

    \item The Terrain HSI data set is constituted of 166 clean bands,
        each having $500 \times 307$ pixels, and is composed of about
        5 different materials: road, tree, bare soil, thin and tick
        grass\footnote{ \url{http://www.way2c.com/rs2.php}}.

\end{itemize}


Tables~\ref{urbanresults}--\ref{terrainresults} show the relative
error attained by the different algorithms on the data sets ``Urban''
($r=6$ and ($r=8$), ``San Diego'' ($r=8$ and $r=10$), and ``Terrain''
($r=5$ and $r=6$).  The reported time refers to the run time of the
algorithms, without the preprocessing step.

\begin{table}[ht]
\begin{center}
\begin{tabular}{|c||c|c||c|c|}
\hline
   &  $r = 6$  &  &  $r = 8$ & \\   \hline
   &  Time (s.)  & Rel.\@ error &  Time (s.)  & Rel.\@ error  \\   \hline
VCA      &  1.02  & 18.05 & 1.05 & 22.68 \\
VCA-100  &  0.05  & 6.67  & 0.07 & 4.76 \\
VCA-500  &  0.03  & 7.19  & 0.09 & 7.25 \\ \hline
SPA      &  0.26  & 9.58 & 0.32 & 9.45 \\
SPA-100  &  $<$0.01  & 9.49  & $<$0.01 & 5.01 \\
SPA-500  &  $<$0.01  & 10.05  & $<$0.01 & 8.86 \\ \hline
SNPA      & 13.60  & 9.63 & 23.02 &  5.64 \\
SNPA-100  &  0.10   & 11.03  &  0.15 & 6.17 \\
SNPA-500  &  0.15  & 10.05 & 0.25 & 8.86  \\ \hline
XRAY     &  28.17 & 7.50 & 95.34 & 6.82 \\
XRAY-100 &  0.11  & 6.78  & 0.17 & 6.57 \\
XRAY-500 &  0.15  & 8.07  & 0.28 & 7.36 \\ \hline
H2NMF    &  12.20 & 5.81  & 14.92 & 5.47 \\
H2NMF-100    & 0.16 & 7.11  & 0.23 & 6.14 \\
H2NMF-500    &  0.27 & 5.87  & 0.37 & 5.68  \\ \hline
FGNSR-100   &  2.73 & 5.58  & 2.55 & 4.62 \\
FGNSR-500   &  40.11 & \textbf{5.07}  & 39.49 & \textbf{4.08} \\ \hline
\end{tabular}
\caption{Numerical results for the Urban HSI (the best result is highlighted in bold).}
\label{urbanresults}
\end{center}
\end{table}

\begin{table}[ht]
\begin{center}
\begin{tabular}{|c||c|c||c|c|}
\hline
   &  $r = 8$  &  &  $r = 10$ & \\   \hline
   &  Time (s.)  & Rel.\@ error &  Time (s.)  & Rel.\@ error  \\   \hline
VCA      &  1.71  &  7.46 & 1.79 & 9.46 \\
VCA-100  &  0.07  &  8.49 & 0.12 & 6.08 \\
VCA-500  &  0.06  &  9.19 & 0.13 & 6.29 \\ \hline
SPA      &  0.53     & 12.62 &    0.61 & 7.01 \\
SPA-100  &  0.03     & 8.49  &    0.01 & 5.83 \\
SPA-500  &  $<$0.01  & 12.64 & $<$0.01 & 6.61 \\ \hline
SNPA      & 38.95 & 12.84 & 58.45 & 7.67 \\
SNPA-100  &  0.22 &  8.49 &  0.20 & 6.90 \\
SNPA-500  &  0.25 & 12.64 &  0.48 & 6.47 \\ \hline
XRAY     &  93.29 & 13.06 & 243.40 & 12.62 \\
XRAY-100 &  0.14  & 8.68  & 0.21 & 5.12 \\
XRAY-500 &  0.19  & 13.17  & 0.35 & 6.82 \\ \hline
H2NMF    &  21.51 & 4.75  & 24.42 & 4.28 \\
H2NMF-100    & 0.30 &  6.85  & 0.22 &  5.61 \\
H2NMF-500    & 0.33  & 6.78   &  0.38 & 5.75 \\ \hline
FGNSR-100   &  2.55 & \textbf{3.73}  & 2.47 & \textbf{3.40} \\
FGNSR-500   &  38.70 & 4.05  & 38.28 & \textbf{3.40} \\ \hline
\end{tabular}
\caption{Numerical results for the San Diego HSI (the best result is highlighted in bold).}
\label{sandiegoresults}
\end{center}
\end{table}

\begin{table}[ht]
\begin{center}
\begin{tabular}{|c||c|c||c|c|}
\hline
   &  $r = 5$  &  &  $r = 6$ & \\   \hline
   &  Time (s.)  & Rel.\@ error &  Time (s.)  & Rel.\@ error  \\   \hline
VCA      &  1.65  & 10.92 & 1.67 & 6.22 \\
VCA-100  &  0.02  &  5.59 & 0.03 & 7.33 \\
VCA-500  &  0.03  &  5.77 & 0.03 & 5.57 \\ \hline
SPA      &  0.38  &  5.89 & 0.43 & 4.81 \\
SPA-100  &  $<$0.01  & 4.74  & 0.01 & 3.95 \\
SPA-500  &  0.01  & 4.83  & 0.01 & 4.63 \\ \hline
SNPA      & 17.54 & 5.76  & 24.28 & 4.60  \\
SNPA-100  &  0.10 & 5.75  & 0.11 & 5.65 \\
SNPA-500  &  0.10 & 4.83 & 0.13  & 4.78 \\ \hline
XRAY     &  33.63 & 5.39 & 73.91 & 5.17 \\
XRAY-100 &  0.07  & 4.15  & 0.12 & 4.13 \\
XRAY-500 &  0.09  & 5.21  & 0.19 & 4.97 \\ \hline
H2NMF    &  18.23 & 5.09  & 20.92 & 4.85 \\
H2NMF-100    & 0.15 & 4.72  & 0.17 & 4.39  \\
H2NMF-500    & 0.23  & 5.43 & 0.29   & 5.35  \\ \hline
FGNSR-100   &  4.23 & \textbf{3.34}  & 2.63 & \textbf{3.21} \\
FGNSR-500   &  40.29 & 3.68  & 40.13 & 3.39 \\ \hline
\end{tabular}
\caption{Numerical results for the Terrain HSI (the best result is highlighted in bold).}
\label{terrainresults}
\end{center}
\end{table}

We observe that FGNSR-100 and FGNSR-500 perform consistently better
than all the other algorithms, although, as expected, at a higher
computational cost than SPA and VCA.  We summarize the results as
follows.
\begin{itemize}
\item For the Urban HSI with $r=6$ (resp.\@ $r=8$),
FGNSR-100 provides a solution with relative error 5.58\% (resp.\@ 4.62\%) and FGNSR-500 with relative error 5.07\% (resp.\@ 4.08\%),
the third best being VCA-100 with 5.94\% (resp.\@ SPA-100 with 5.01\%).

\item For the San Diego airport HSI with $r=8$ (resp.\@ $r=10$),
FGNSR-100 provides a solution with relative error 3.73\% (resp.\@ 3.40\%)
and FGNSR-500 with relative error 4.05\% (resp.\@ 3.40\%),
the third best being H2NMF with 4.75\% (resp.\@ XRAY-100 with 5.12\%).

\item For the Terrain HSI with $r=5$ (resp.\@ $r=6$), FGNSR-100 provides a solution with relative error 3.34\% (resp.\@ 3.21\%)
and FGNSR-500 with relative error 3.68\% (resp.\@ 3.39\%),
the third best being XRAY-100 with 4.15\% (resp.\@ SPA-100 with 3.95\%).
 \end{itemize}

It is interesting to note that, in most cases, near-separable algorithms applied on the subset of columns identified by H2NMF
perform much better than when applied
on the full data set.
The reason is that these algorithms tend to extract outlying pixels which are filtered out by the subsampling procedure (especially when the number of clusters is small).

\opt{ieee}{

We show the computed endmembers for the ``Urban'' data set in
Figure~\ref{fig:Urban_endmembers}, and the corresponding abundance
maps for each of the species is shown in
Figures~\ref{fig:Urban_abmap}.

\begin{figure}[t]
    \includegraphics[width=0.9\columnwidth]{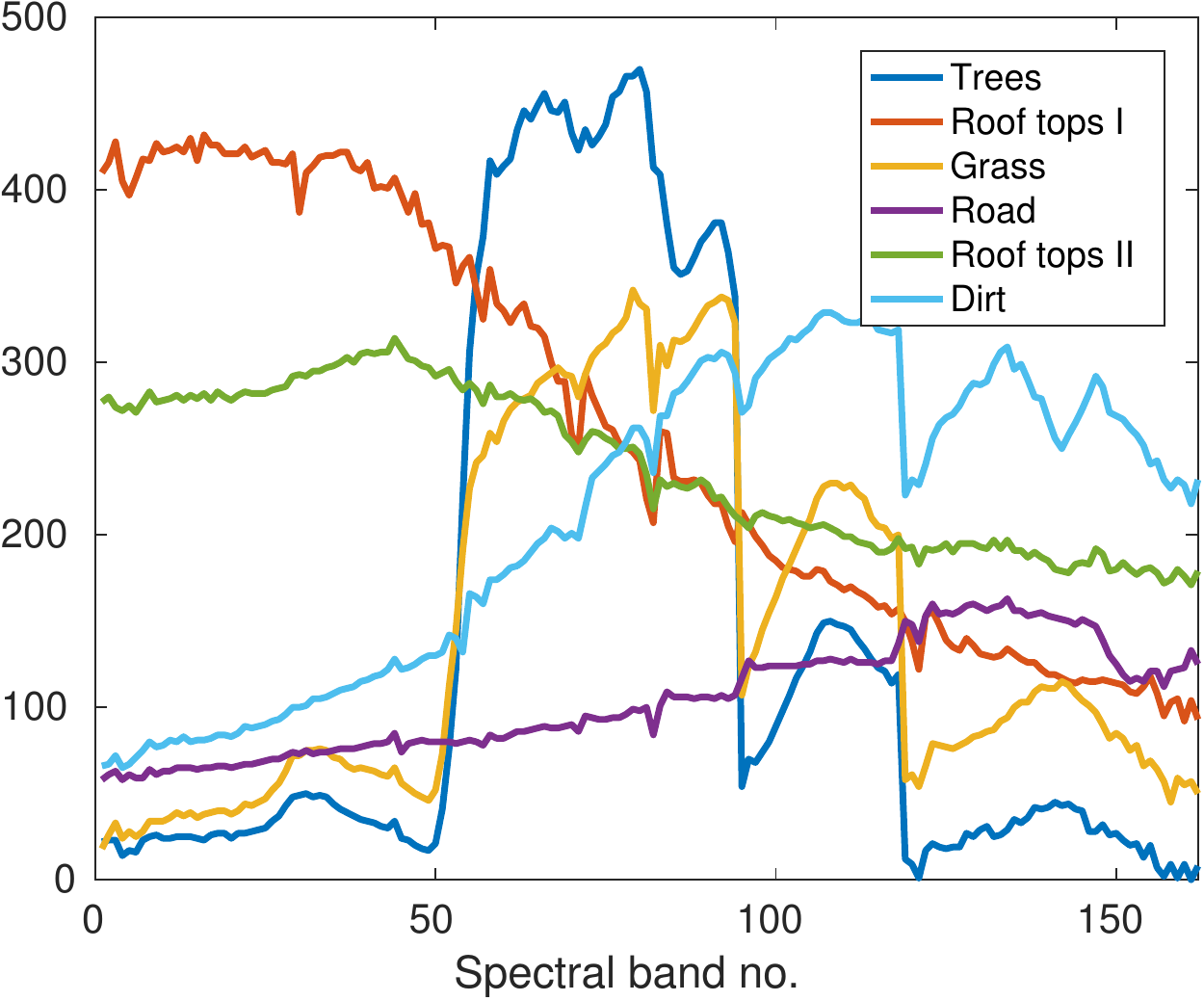}
    \caption{Endmembers of the Urban dataset obtained from
    FGNSR after preprocessing with $500$
    clusters\label{fig:Urban_endmembers}}
\end{figure}

\begin{figure*}[t]
    \begin{center}
    \includegraphics[width=0.99\textwidth]{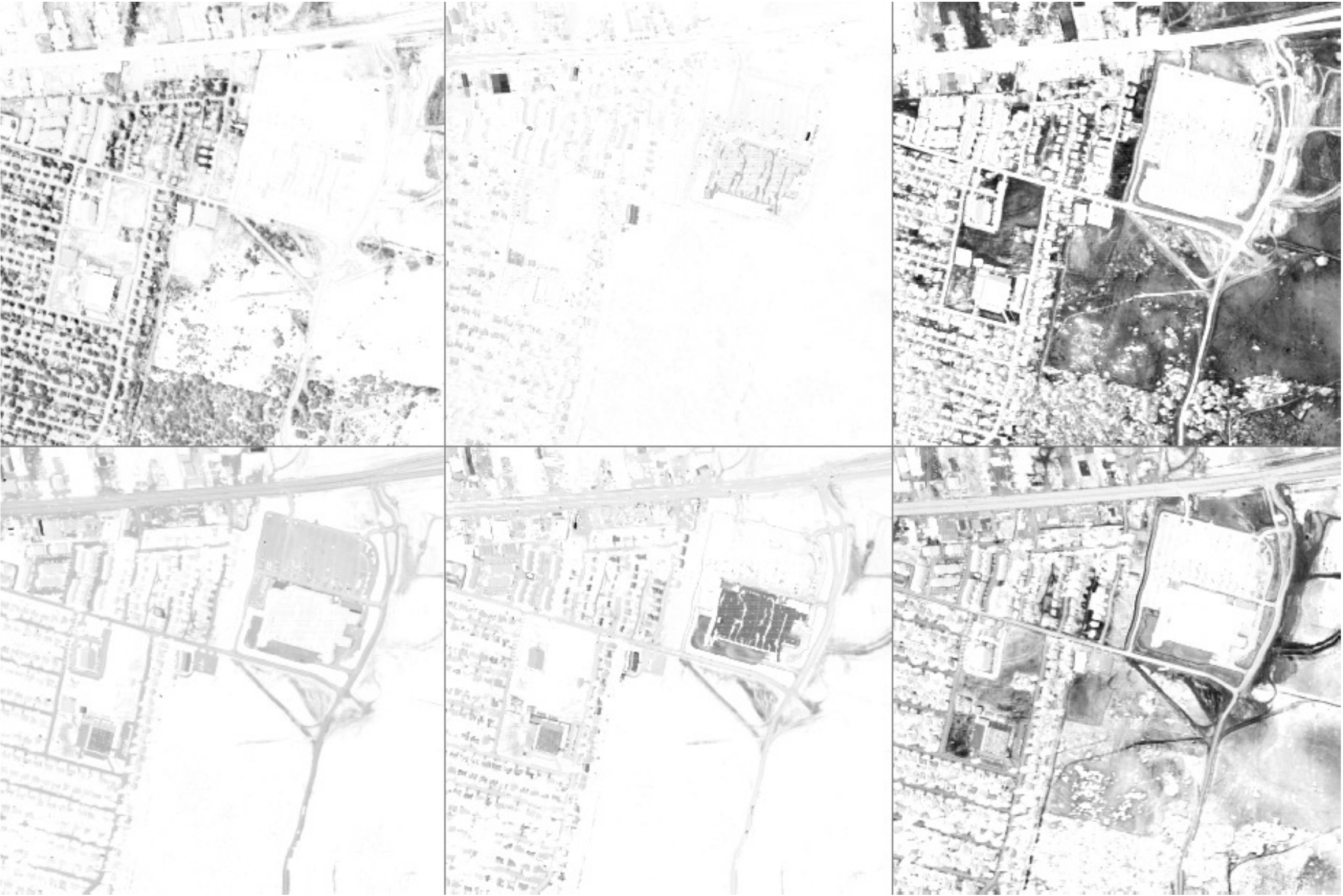}
\caption{Abundance maps corresponding to the endmembers extracted by FGNSR-500 for the Urban HSI ($r=6$). From left to right, top to bottom:
(i)~trees, (ii)~roof tops I, (iii)~grass, (iv)~road, (v)~roof tops II, (vi)~dirt.
\label{fig:Urban_abmap}}
    \end{center}
\end{figure*}

}

\opt{preprint}{
\begin{figure*}[p]
\begin{center}
\includegraphics[height=0.3\textheight]{Urbanend}\\
\includegraphics[height=0.3\textheight]{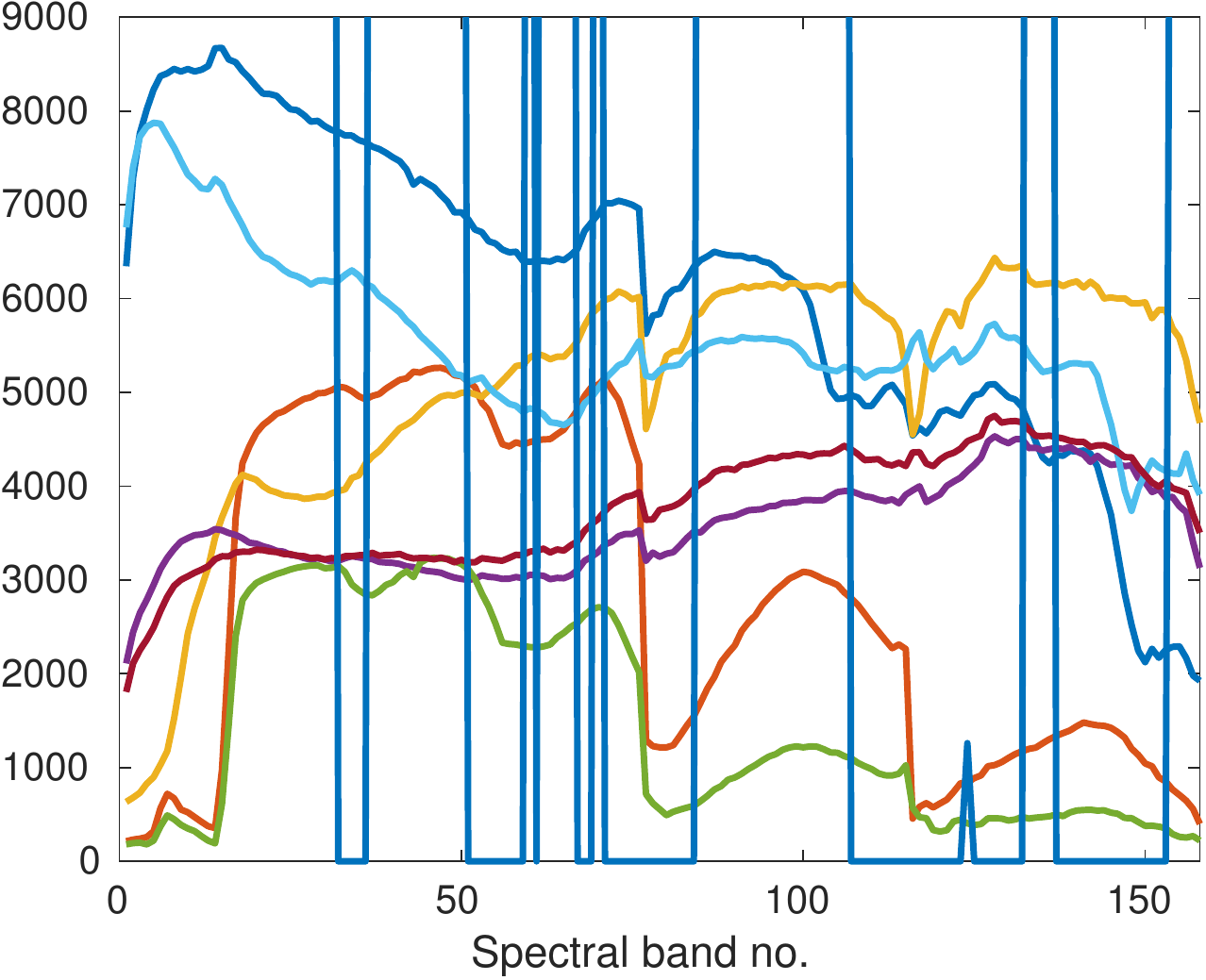}\\
\includegraphics[height=0.3\textheight]{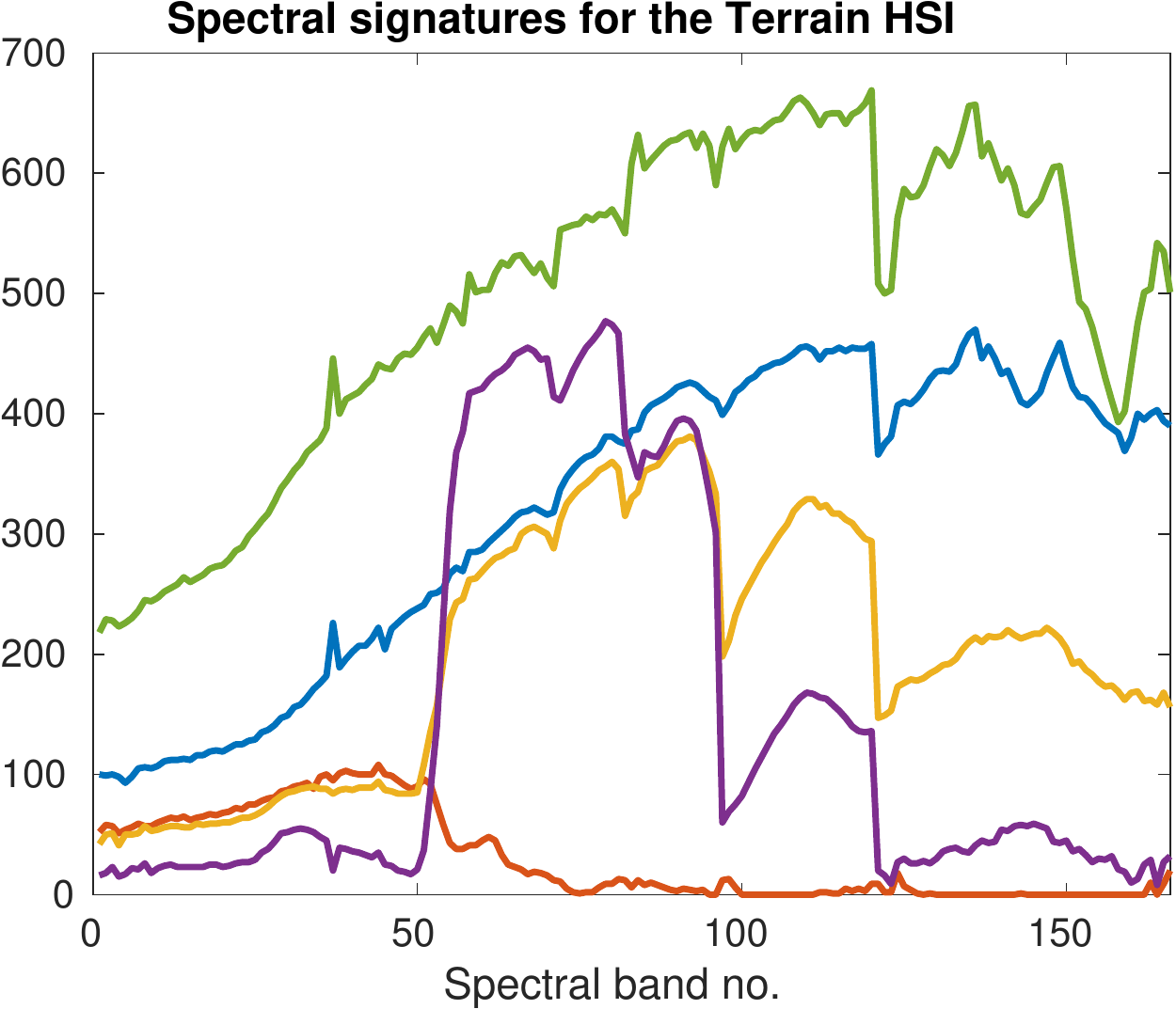}
\caption{Endmembers extracted by FGNSR. From top to bottom:
Urban HSI (FGNSR-500, $r=6$),
San Diego airport HSI (FGNSR-100, $r=8$), and
Terrain HSI (FGNSR-100, $r=5$).   The endmembers of the San Diego
    airport data set include one outlier, which corresponds to the
    truncated, dark blue signature.
\label{fig:endmembers}}
\end{center}
\end{figure*}

Figure~\ref{fig:endmembers} displays the endmembers extracted by FGNSR
for the three HSI's.  We observe that the extracted endmembers are
well separated. Note that for the San Diego airport data set, one
outlier is extracted. The reason is that the norm of its spectral
signature is extremely large (with values up to 32753) hence it has to
be extracted to reduce the error to a low value.

\begin{figure*}[p]
\includegraphics[width=\textwidth]{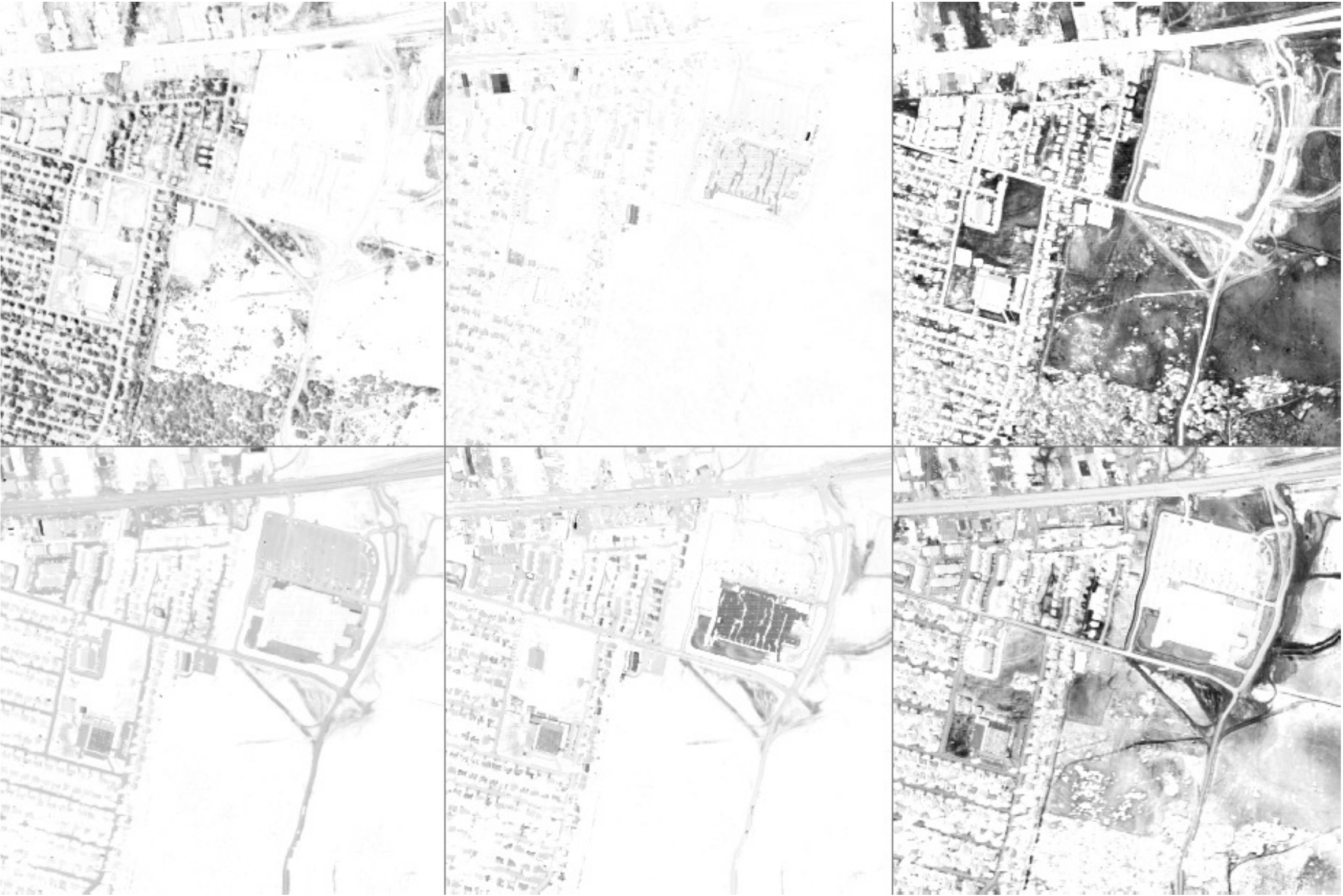}
\caption{Abundance maps corresponding to the endmembers extracted by FGNSR-500 for the Urban HSI ($r=6$). From left to right, top to bottom:
(i)~trees, (ii)~roof tops I, (iii)~grass, (iv)~road, (v)~roof tops II, (vi)~dirt.
\label{fig:Urbanabmap}}
\end{figure*}

\begin{figure*}[p]
\includegraphics[width=\textwidth]{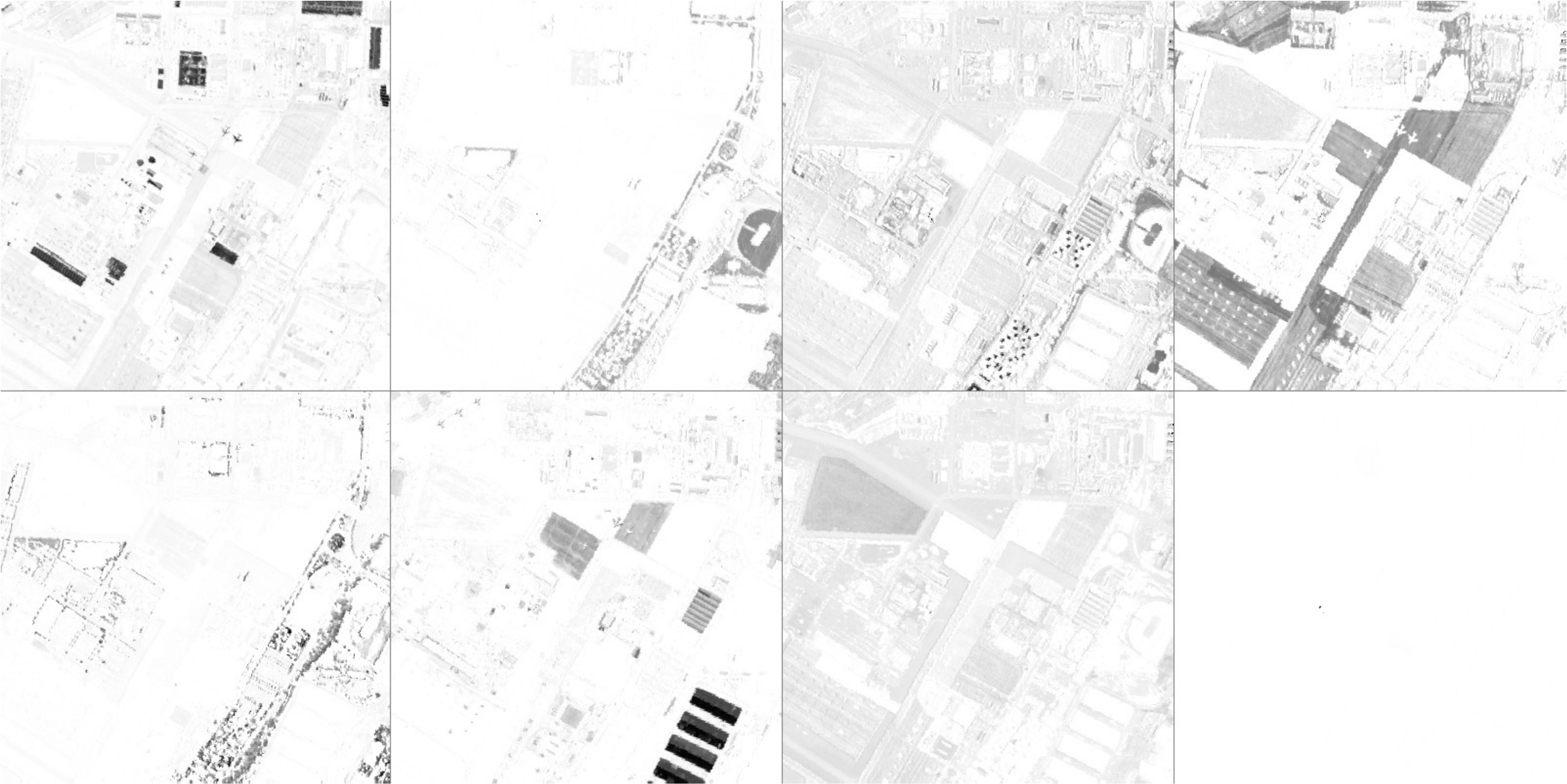}
\caption{Abundance maps corresponding to the endmembers extracted by FGNSR-100 for the San Diego airport HSI ($r=8$).
From left to right, top to bottom:
(i)~roof tops I, (ii)~grass, (iii)~dirt, (iv)~road surface I, (v)~trees, (vi)~roof tops II, (vii)~road surface I, (viii)~outlier.
\label{fig:SanDiegoabmap}}
\end{figure*}

\begin{figure*}[t]
\includegraphics[width=\textwidth]{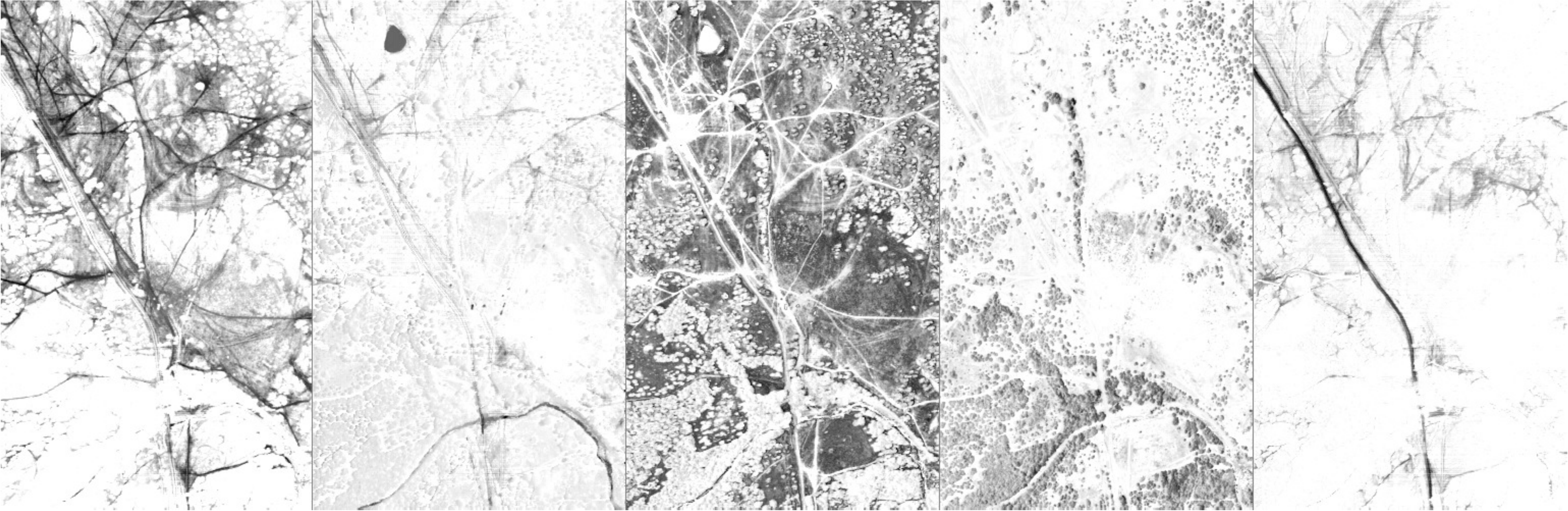}
\caption{Abundance maps corresponding to the endmembers extracted by FGNSR-100 for the Terrain HSI ($r=7$).
From left to right, top to bottom:
(i)~roads I, (ii)~roads II, (iii)~grass, (iv)~trees, (v)~roads III.
\label{fig:Terrainabmap}}
\end{figure*}

Figures~\ref{fig:Urbanabmap},~\ref{fig:SanDiegoabmap},
and~\ref{fig:Terrainabmap} display the abundance maps corresponding to
FGNSR. We observe that the abundace maps are relatively well separated
and sparse, which confirms the quality of the endmembers extracted by
FGNSR (the only constraint imposed on the weights $H$ is
nonnegativity).

}  





To conclude, we have observed on three data sets that FGNSR-100 and FGNSR-500 are able to identify the best subset of columns to reconstruct the original input image in all cases, while its computational cost is reasonable.

		\begin{remark}
	Note that we also compared the algorithms on the widely used Cuprite data set but the results are not very interesting as most algorithms find very similar solutions in terms of relative error.
	The reason is that the data is not contaminated with outliers and spectral signatures are rather similar in  that data set.
	For example, for $r = 15$, all algorithms have relative error in the interval [1.39,1.99]\%, and, for $r = 20$, in the interval [1.35,1.84]\%.
	\end{remark} 


  \section{Conclusion and Further Work}

	In this paper, we analyzed a robust convex optimization model for dealing with nonnegative sparse regression with self dictionary; in particular showing its close connection with the model proposed in~\cite{EMO12}. We then developed an optimal first-order method to solve the problem. 	We showed that this approach outperforms standard near-separable NMF algorithms on synthetic data sets, and on real-world HSI's when  combined with a hierarchical clustering strategy. Moreover, we observed that preselecting a small number of good candidate allows all near-separable NMF algorithms to perform much better.

The model~\eqref{GLqp2} (and the corresponding
Algorithm~\ref{alg:fgnsr}) can be easily generalized to handle any
dictionary $D$, changing the model to
\begin{align*}
    \min_{X \in \Omega'} \; \trace(X) & \quad
        \text{such that}  \quad \norm{M - DX}_F \leq \epsilon,
\end{align*}
where
\begin{equation*}
    \Omega' = \{ X \in [0,1]^{n,n} \setcond
        X_{ij} \norm{D(:,i)}_1 \leq X_{ii} \norm{M(:,j)}_1  \forall i,j \}.
\end{equation*}
It would therefore be an interesting direction of further research to
analyze and apply this model in other contexts.

Further it would be of interest to study the robustness of
Algorithm~\ref{alg:fgnsr} if combined with random projections
(see~\cite{ding2013topic,benson2014scalable}), and to consider
non-additive noise models, e.g., spectral variability in the context of
HSI~\cite{zare2014endmember}.

\opt{ieee}{\appendices}
\opt{preprint}{\begin{appendix}}

\section{Proof of Theorem~\ref{thm:equiv}}
\label{sec:proof_equiv}
Observe that
\begin{itemize}

    \item Any feasible solution of~\eqref{GLmod} is a feasible
        solution of~\eqref{eq:esser}: in fact, the only difference
        between the feasible domains are the additional constraints
        $X_{ij} \leq X_{ii}$ for all $i, j$.


    \item The objective function of~\eqref{eq:esser} is larger than the
        one of \eqref{GLmod}: in fact, by definition, $\trace(X) \leq
        ||X||_{1,\infty}$. Moreover, the two objective functions
        coincide if and only if $X_{ii} = \max_{j} X_{ij}$ for all
        $i$.

\end{itemize}

These observations imply that the optimal objective function value
of~\eqref{GLmod} is larger than the one of~\eqref{eq:esser} (since
any feasible solution $X$ of~\eqref{GLmod} is feasible
for~\eqref{eq:esser} and satisfies $\trace(X) = \norm{X}_{1,\infty}$).

Therefore, if we can transform any optimal solution $X^*$
of~\eqref{eq:esser} into a feasible solution $X^{\dagger}$
of~\eqref{GLmod} with the same objective function value,
$X^{\dagger}$ will be an optimal solution of~\eqref{GLmod} and the
proof will be complete.

Let  $X^*$ be any optimal solution of~\eqref{eq:esser}.  If $X^* = 0$,
then $X^*$ is trivially feasible for~\eqref{GLmod} and the proof
is complete.

So assume $X^* \neq 0$.  We will show by contradiction that
$\norm{M-MX^*} = \epsilon$, so assume that $\norm{M - MX^*} <
\epsilon$.  By continuity of norms there exists $0 < \delta < 1$
so that $\norm{M - M (\delta X^*)}  < \epsilon$. The matrix $\delta
X^*$ is a feasible solution for \eqref{eq:esser} since $0 \leq \delta X^* \leq
X^* \leq 1$ while $\norm{\delta X^*}_{1,\infty} = \delta
\norm{X^*}_{1,\infty} < \norm{X^*}_{1,\infty}$, a contradiction to the
optimality of $X^*$.

Assume that $X^*_{ii} < X^*_{ij} \leq 1$ for some $j$. Let us show
this is only possible if $M(:,i) = M X^*(:,i)$: assume $M(:,i)
\neq M X^*(:,i)$, we have
\begin{equation*}
    M(:,i) - M X^*(:,i) = (1 - X^*_{ii})  M(:,i)  - M X^*(\mathcal{I},i),
\end{equation*}
where $\mathcal{I} = \{1,\dotsc,n\} \setminus \{i\}$.  Increasing
$X^*_{ii}$ to $X^*_{ij}$ while decreasing the entries of
$X^*(\mathcal{I},i)$ by the factor $\beta =
\frac{X^*_{ii}}{X^*_{ij}} < 1$ decreases $\norm{M(:,i) - M X(:,i)}_c$
by a factor $(1 - X^*_{ii})$ which would be a contradiction since
$\norm{M-MX}$ would be reduced (see above).

Finally, let us construct another optimal solution $X^{\dagger}$: we
take $X^{\dagger} = X^*$, and for all $j$ such that $X^*_{ii} <
X^*_{ij} \leq 1$, $X^{\dagger}_{ii}$ is replaced with
$X^{\dagger}_{ij}$ and $X^{\dagger}(\mathcal{I},i)$ is multiplied
by the factor $\beta = \frac{X^*_{ii}}{X^*_{ij}} < 1$.  The error
$\norm{M-MX}$ remains unchanged while the objective function might
have only decreased: $X^{\dagger}$ is an optimal solution of
\eqref{eq:esser} satisfying $X^{\dagger}_{ii} = \max_j
X^{\dagger}_{ij}$ hence is also an optimal solution of
\eqref{GLmod}.


\section{Projection onto $\Omega$}
\label{sec:projection}

We now give the details for evalutating the Euclidean projection onto
the set~$\Omega$, see Section~\ref{sec:proj}. Note first that it is
sufficient to consider the problem of projecting  a single row of $X$,
say, the first one, on the set
\begin{equation*}
    \Omega_1 \defby \{ z \in \Rnn^n \setcond
        z_1 \le 1, w_1 x_j \le w_j z_1 \},
\end{equation*}
since the the rows of $X$ can be projected individually on $\Omega$ as
they do not depend on each other in $\Omega$.  Further we may assume
that $w > 0$ because $w_1 = 0$ removes all constraints on $x_j$ for $2
\le j \le n$, and $w_j$ for $j \neq 1$ fixes $x_j = 0$ for any point
in $\Omega_1$.  Similarly we can assume w.l.o.g.\@ that $x_j > 0$ for $2
\le j \le n$, since $x_j \le 0$ fixes $z_j = 0$ for any point in
$\Omega_1$ and does not affect the choice of the first coordinate.
Note that the norm we wish to minimize now is given by the standard
Euclidean norm $\norm{x} = (\sum_j x_j^2)^{\frac{1}{2}}$.

Let $t \in \R$ be a parameter and denote by $\phi_t : \R^n
\rightarrow \R^n$ the mapping
\begin{equation}
\label{eqn:phi_t}
    \phi_t(x)_j \defby
    \begin{cases}
        t & \text{if}\; j=1,\\
        \tfrac{w_j}{w_1} t & \text{if $j\neq 1$ and} \;
            \tfrac{w_1}{w_j} x_j \ge t,\\
        x_j & \text{else.}
    \end{cases}
\end{equation}
From the definition we see that for all $x \in \R^n$ and all $t \in
[0,1]$ we have that $\phi_t(x) \in \Omega_1$.

Unfortunately we cannot simply assume that $0 \le x_1 \le 1$.
Treating the cases $x_1 < 0$, $0 \le x_1 \le 1$ and $1 < x_1$
homogeneously puts a burden on the notation and slightly obfuscates the
arguments used in the following.  We set $x_1^+ \defby \min \{1, \max
\{ 0, x_1 \} \}$ and $x_1^- \defby \min \{0, x_1 \}$.

\begin{lemma}
\label{lem:univariate}
Let $x \in \R^n$, then
\begin{equation*}
    \min_{z \in \Omega_1} \norm{x - z} =
    \min_{t \in [x_1^+, 1]} \norm{x - \phi_t(x)}.
\end{equation*}
In particular, if $z^* = \argmin_{z \in \Omega_1} \norm{x -z}$, we have for
$2 \le j \le n$ that
\begin{align*}
    w_1 z^*_j & < w_j z^*_1 \quad \text{if} \; w_1 x_j < w_j z^*_1\\
\intertext{and}
    w_1 z^*_j & = w_j z^*_1 \quad \text{if} \; w_1 x_j \ge w_j z^*_1.
\end{align*}
\end{lemma}
\begin{proof}
Let $z^* = \argmin_{z \in \Omega_1}$.  We will show first that
\begin{equation*}
    \min_{z \in \Omega_1} \norm{x - z} =
    \min_{t \in [0, 1]} \norm{x - \phi_t(x)}.
\end{equation*}

If $w_1 x_j < w_j z^*_1$, it follows $z^*_j = x_j$,
because only in this case the cost contribution of the $j$-th
coordinate is minimum (zero).  Otherwise, if $w_1 x_j \ge w_j z^*_1$,
the projection cost from the $j$-th coordinate is minimized only if
$x_j = \tfrac{w_j}{w_1} z^*_1$, because all the quantities involved
are positive.  It follows that $z^*_j = \phi_{z^*_1}(x)_j$ for $2 \le
j \le n$. And since $z^*_1 \in [0,1]$, it follows that $z^* =
\argmin_{t \in [0,1]} \norm{x - \phi_t(x)}$.

To finish the proof, we will now show that $z_1^* \ge x_1^+$, which is
trivial for the case $x_1 < 0$. In the case $x_1 \in [0,1]$, we show
$z_1^* \ge x_1$.  In order to obtain a contradiction, assume $z_1^* <
x_1$, and define $\tilde{z} \in \Omega_1$ by
\begin{equation*}
    \tilde{z}_j =
    \begin{cases}
        x_1   & \text{if $j = 1$}\\
        z^*_j & \text{if $j \neq 1$}.
    \end{cases}
\end{equation*}
We compute
\begin{equation*}
\begin{split}
    \norm{x - \tilde{z}}_2
    & = \sum_{j=1}^n (x_j - \tilde{z}_j)^2
    = \sum_{j=2}^n (x_j - z^*_j)^2
    < \sum_{j=1}^n (x_j - z^*_j)^2\\
    & = \norm{x - z^*}_2,
\end{split}
\end{equation*}
which contradicts the optimality of $z^*$.

In the case $x_1 > 1$, a similar reasoning shows that $z_1^* \ge 1$
(in fact $z_1^* = 1$).
\end{proof}

The previous lemma shows that the optimal projection can be computed
by minimization of a univariate function.  We will show next that this
can be done quite efficient.  For a given $x \in \R^n$ we define the
function
\begin{equation*}
   c_x  : [x_1^-, \infty [ \rightarrow \R, \quad
   t  \mapsto \norm{x - \phi_t(x)}^2.
\end{equation*}
By Lemma~\ref{lem:univariate}, the (squared) minimum projection cost
is given by the minimum of $c_x\restrict_{[x_1^+,1]}$, and our next
step is to understand the behavior of $c_x$ (see
Fig.~\ref{fig:euclproj}).

In order to simplify the notation in the following, we abbreviate $b_j
\defby \frac{w_1}{w_j} x_j$
and assume that the components of $x$ are ordered so that
$b_2 \le b_3 \le \dotsb \le b_n$, and since $x_j > 0$ for $2 \le j
\le n$, we have in fact that
\begin{equation}
\label{eqn:ordering}
    x_1^- \bydef b_1 \le 0 < b_2 \le b_3 \le \dotsb \le b_n.
\end{equation}

\begin{lemma} Let $x \in \R^n$.
\begin{compactenum}[(i)]
\item \label{case:convex}
The function $c_x$ is a piecewise $C^\infty$ function with $C^1$
break points $b_j$, and each piece  $c_x \restrict_{[b_k, b_{k+1}]}$
is strongly convex.  In particular, $c_x$ is strongly convex and
attains its minimum.

\item \label{case:pullback}
Let $t^* = \argmin_{t \in [x_1^-, \infty [} c_x(t)$, the optimal
projection of $x$ on $\Omega_1$ is
\begin{equation}
\label{eqn:pullback}
    \argmin_{z \in \Omega_1} \norm{x - z} =
    \begin{cases}
        \phi_{x_1^+}(x) & \text{if $t^* < x_1^+$},\\
        \phi_{t^*}(x)   & \text{if $t^* \in [x_1^+,1]$ or}\\
        \phi_{1}(x  )   & \text{if $1 < t^*$}.
    \end{cases}
\end{equation}
\end{compactenum}
\end{lemma}
\begin{proof}
We show~(\ref{case:convex}) first.  Using the
definition~\eqref{eqn:phi_t}, we compute for $x \in \R^n$ and $t
\in [b_1, \infty[$
\begin{gather*}
    c_x(t)
    = \norm{x - \phi_t(x)}^2
    = (x_1 - t)^2 + \sum_{j \in B(t)} (x_j - \tfrac{w_j}{w_1}t)^2,
\end{gather*}
where
\begin{equation*}
    B(t) \defby \{ 1 \le j \le n \setcond t \le b_j \},
\end{equation*}
from which we see that $c_x$ is piecewise smooth and continuously
differentiable at the points $b_j$.  For a point $t \in ]b_k,
b_{k+1}[$, we see that $c_x''(t) > 0$ which shows strong convexity on
each piece.  But $c_x'$ is continuous at each $b_j$, so $c_x$ is
strongly convex on all of $[b_1, \infty[$.  Finally, since $\lim_{t
\rightarrow \infty} c_x(t) = \lim_{t \rightarrow \infty} (x_1 - t)^2 =
\infty$, we see see that $c_x$ attains its minimum.

The assertions in~(\ref{case:pullback}) follow directly from
Lemma~\ref{lem:univariate} and the convexity of~$c_x$.
\end{proof}

Since $c_x$ attains its minimum at $t^* \in [b_1, \infty[$, we have
that either $t^* \in [b_1, b_2]$, $t^* \in ]b_k, b_{k+1}]$, for some
$2 \le k < n$, or $t^* \in ]b_n, \infty[$.  By the
definition of $\phi_t$, the constraints $w_1 x_j \le w_j x_1$
corresponding to break points $b_j \ge t^*$ are ``active'' at the
point $\phi_{t^*}(x)$, while all other constraints are not active.
More formally, in each of the cases in~\eqref{eqn:B_star}, we can
uniquely associate a set of optimal active indices $B^* \subseteq \{2,
\dotsc, n\}$ as follows (here $2 \le k \le n-1$):
\begin{equation}
\label{eqn:B_star}
    B^* =
    \begin{cases}
        B_1 \defby \{ 2, \dotsc, n \}
            & \text{iff $t^* \in [b_1, b_2] \bydef T_1$}\\
        B_k \defby \{ k+1, \dotsc, n \}
            & \text{iff $t^* \in ]b_k, b_{k+1}] \bydef T_k$}\\
        B_n \defby \emptyset
            & \text{iff $t^* \in ]b_n, \infty[ \bydef T_n$}.
    \end{cases}
\end{equation}

The preceding observation yields an efficient ``dual'' algorithm for
minimizing $c_x$ over $[b_1, \infty[$, which we develop in the
following lemma.

\begin{lemma}
Let $x \in \R^n$, and $t^* = \argmin_{t \in [b1, \infty[} \norm{x -
\phi_t(x)}$,
denote the set of active indices corresponding to $t^*$ by $B^* = \{1
\le j \le n \setcond b_j \ge t^* \}$, and set
\begin{equation}
\label{eqn:t_formula}
    t_k = w_1 \frac{w_1 x_1 + \sum_{j \in B_k} w_j x_j}
         {w_1^2 + \sum_{j \in B_k} w_j^2},
    \quad 1 \le k \le n.
\end{equation}
If $t_k \in T_k$, then $t^* = t_k$ (and $B^* = B_k$).
\end{lemma}
\begin{proof}
Let $1 \le k \le n$.  In order to minimize the strongly convex
function
\begin{equation*}
    c_x(t) = (x_1 - t)^2 +
      \sum_{j \in B_k} (x_j - \tfrac{w_j}{w_1} t)^2,
\end{equation*}
one simply solves $c_x'(t) = 0$ for $t$, and obtains
expression~\eqref{eqn:t_formula}.  So $t_k$ is the minimum of $c_x$
under the hypothesis that $B_k$ is the correct guess for $B^*$.  But
by~\eqref{eqn:B_star} we have that $B_k = B^*$ if, and only if, $t_k
\in T_k$, which gives a trivially verifiable criterion for deciding
whether the guess for $B^*$ is correct.
\end{proof}

\begin{figure}
\begin{center}
\includegraphics[width=.48\textwidth]{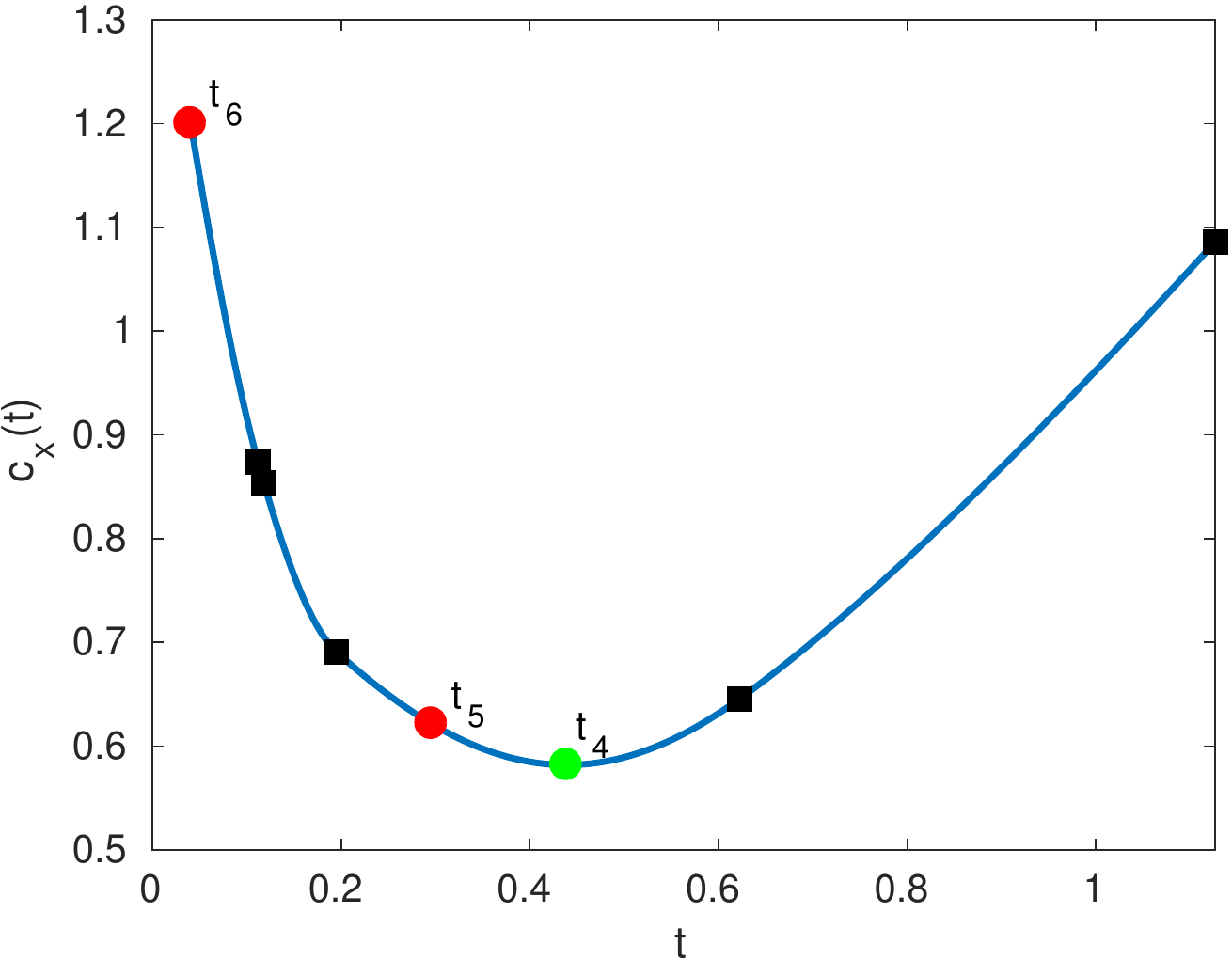}
\end{center}
\caption{Example for the minimization of $c_x$ ($n=6$).  Black squares
indicate the break points and the discs show the location of the trial
values $t_k$ from~\eqref{eqn:t_formula}.  The optimal projection is
realized by $t_4$ and the set of optimal active indices are the two
break points to the right of $t_4$.\label{fig:euclproj}}
\end{figure}

The algorithmic implication of this lemma is as follows.  Since we
know that one of the sets $B_1, \dotsc, B_n$ must be the optimal
active set $B^*$, we simply compute for each such set $B_k$ the
corresponding optimal point from~\eqref{eqn:t_formula} until we
encounter a point $t_k \in T_k$, which then is the sought optimum.

\begin{theorem}
The Euclidean projection of $X \in R^{n,n}$ on $\Omega$ can be
computed in $\bigo{n^2 \log n}$.
\end{theorem}
\begin{proof}
The projection cost for each row of $X$ amounts to
evaluating~\eqref{eqn:t_formula} for each of the sets $B_k$ until the
optimal set of active indices is identified.  If the sets $B_k$ are
processed in the ordering $B_n, B_{n-1}, \dotsc, B_1$, the nominator
and denominator in~\eqref{eqn:t_formula} can be updated from one set to
the next, resulting in a computation linear in $n$.  The only
non-linear cost per row is induced by sorting the break points of
$c_x$ as in~\eqref{eqn:ordering}, which can be done in $\bigo{n\log n}$, and
results in the stated worst-case complexity bound.
\end{proof}

The function $c_x$ is shown in Fig.~\ref{fig:euclproj} for a randomly
chosen vector $x$ and weights $w$.  In this example, the sets $B_6,
B_5$ and $B_4$ are tested for optimality; the corresponding values
$t_k$ are indicated in the plot.

\begin{remark} \label{rem:datastruct} In an implementation of the
    outlined algorithm it is not necessary to sort \emph{all} the
    break points of $c_x$ for a row $x$ of $X$ as
    in~\eqref{eqn:ordering}.  Only the break points $b_j \in [x_1^+,
    1]$ need to be considered, as all other break points are either
    never (if $b_j < x_1^+$) or always (if $b_j > 1$) in the optimal
    set $B^*$ of active indices.  Denote $k_1$ the number of break
    points in $[x_1^+,1]$ and $k_2 \defby \card{B^*}$.  If the indices
    are not sorted but maintained on a heap~(see,
    e.g.,~\cite{Tarjan1983}), the cost overhead for sorting is reduced
    to $\bigo{k_2 \log k_1}$.  Hence the overall complexity for
    projecting a single row is $\bigo{n + k_2 \log k_1}$.  The
    resulting algorithm is sketched in Algorithm~\ref{alg:projection}.
    Asymptotically it still has the worst case complexity of $\bigo{n
    \log n}$ as we may need to extract all $n$ possible elements from
    $h$, but it should run much faster in practice.
\end{remark}

\begin{algorithm}
\caption{Euclidean projection on $\Omega_1$} \label{alg:projection}
\begin{algorithmic}[1]
\REQUIRE $x \in \R^n$ with $x_2,\dotsc,x_n > 0$, $0 < w \in \R^n$,
heap data structure $h$
\ENSURE $z \in \Omega_1$ with $\norm{x - z}_2$ minimal
\STATE $x_1^+ \leftarrow \max \{ 0, x_1 \}$
\STATE $\mathcal{B} \leftarrow \{ 2 \le j \le n \setcond
    x_1^+ \le \tfrac{w_1}{w_j} x_j \le 1 \}$
\STATE $\mathcal{B}^* \leftarrow \{2 \le j \le n \setcond
    \tfrac{w_1}{w_j} x_j > 1 \}$
\STATE $p \leftarrow w_1 x_1 + \sum_{j \in \mathcal{B}^*} w_j x_j$
\STATE $q \leftarrow w_1^2 + \sum_{j \in \mathcal{B}^*} w_j^2$
\STATE $t \leftarrow w_1 \tfrac{p}{q}$
\STATE initheap($h$, $\{\tfrac{w_1}{w_j}x_j \setcond j \in
    \mathcal{B}\}$) \emph{\% Operation linear in $\card{\mathcal{B}}$.}
\WHILE{$h \neq \emptyset$ and $t < $ findmin($h$)}
    \STATE $M \leftarrow$ extractmin($h$)
    \STATE Let $2 \le j \le n$ such that $M$
        corresponds to $\tfrac{w_1}{w_j}x_j$
    \STATE $\mathcal{B}^* \leftarrow \mathcal{B}^* \cup \{j\}$
    \STATE $p \leftarrow p + w_j x_j$
    \STATE $q \leftarrow q + w_j^2$
    \STATE $t \leftarrow w_1 \tfrac{p}{q}$
\ENDWHILE
\STATE $z \leftarrow x$
\STATE $z_1 \leftarrow \min\{ 1, \max\{0, t\} \}$
    \COMMENT{See~\eqref{eqn:pullback}}
\FORALL{$j \in \mathcal{B}^*$}
    \STATE $z_j \leftarrow z_1 \tfrac{w_j}{w_1}$
        \COMMENT{See~\eqref{eqn:phi_t}}
\ENDFOR
\end{algorithmic}
\end{algorithm}

\opt{preprint}{
    \end{appendix}
}

\opt{preprint}{
\bibliographystyle{siamplain}
}
\opt{ieee}{
\bibliographystyle{IEEEtran}
}
\bibliography{IEEEabrv,FastGradSepNMF}

\end{document}